\documentclass[11pt
]{amsart}
\usepackage{amssymb}
\usepackage{graphicx,color}
\usepackage{hyperref}
\usepackage{fullpage} 

\theoremstyle{plain}

\begin{document}


\theoremstyle{plain}
\newtheorem{theorem}{Theorem} [section]
\newtheorem{corollary}[theorem]{Corollary}
\newtheorem{lemma}[theorem]{Lemma}
\newtheorem{proposition}[theorem]{Proposition}


\theoremstyle{definition}
\newtheorem{definition}[theorem]{Definition}
\theoremstyle{remark}
\newtheorem{remark}[theorem]{Remark}

\numberwithin{theorem}{section}
\numberwithin{equation}{section}

\title{On a fractional Monge-Amp\`ere operator}

\thanks{Luis Caffarelli has been supported by NSF DMS-1540162. Fernando Charro partially supported by a MEC-Fulbright and Juan de la Cierva fellowships and MINECO grants MTM2011-27739-C04-01 and MTM2014-52402-C3-1-P, and is part of the Catalan research group 2014 SGR 1083 (Ògrup reconegutÓ)}

\author[L. Caffarelli]{Luis Caffarelli}

\address{Department of Mathematics, The University of Texas, 1 University Station C1200, Austin, TX 78712}

\email{caffarel@math.utexas.edu}

\author[F. Charro]{Fernando Charro}

\address{Universitat Polit\`ecnica de Catalunya, Departament de Matem\`atica Aplicada I,
 Diagonal 647, 08028 Barcelona, Spain}
 
\email{fernando.charro@upc.edu}

\date{}
\
\maketitle

\begin{abstract}
In this paper we consider a fractional analogue of the Monge-Amp\`ere operator. Our operator is a concave envelope of fractional linear operators  of the form
$
\inf_{A\in \mathcal{A}}L_Au,
$
 where the set of operators is a degenerate class that corresponds to  all affine transformations of determinant one of a given multiple of the fractional Laplacian.

We set up a relatively simple framework of global solutions prescribing data at infinity and global barriers. In our key estimate, we show that the operator remains strictly elliptic, which allows to apply known regularity results for uniformly elliptic operators and deduce that solutions  are classical.
\end{abstract}

\section{Introduction}\label{section.intro}
The classical Monge-Amp\`ere equation $\det D^2u=f$ arises in many areas of analysis, geometry, and applied mathematics. Standard boundary value problems are the Dirichlet problem and optimal transportation problem, where we prescribe the image of a domain by the gradient map.  

In the  Dirichlet problem, one prescribes a smooth domain $\Omega\subset\mathbb{R}^n$, boundary data $g(x)$ on $\partial \Omega$, and a right-hand side $f(x)$ in $\Omega$, and studies  existence and regularity of a function $u$ such that
\begin{equation}\label{problem.intro.model}
\left\{
\begin{split}
& \det D^2 u=f \quad\textrm{in}\ \Omega\\
& u=g\quad\textrm{on}\ \partial \Omega.
\end{split}
\right. 
\end{equation}

For the problem to fit in the framework of the theory of fully nonlinear elliptic equations, one must seek convex solutions $u$ to ensure that $\det D^2 u$ is
indeed a monotone function of $D^2u$. Thus, we must require  $f(x)$  positive and $\Omega$  convex. The convexity of $\Omega$ is required in order to construct appropriate smooth subsolutions that act as lower barriers, see \cite{Ca-Ni-Sp2}.

In that case, there is considerable work in the literature establishing the existence, uniqueness and regularity of solutions to \eqref{problem.intro.model}, see \cite{ Ca-Ni-Sp, Ca-Ni-Sp2, Gutierrez, Trudinger.Wang} and the references therein. The main ingredients entering  the theory are, roughly speaking, the following:
\begin{enumerate}
 \item[(a)] The Monge-Amp\`ere equation is a concave fully nonlinear equation.
For a convex solution
\[
\det D^2u=f
\]
is equivalent to
\[
\inf_{L\in \mathcal{A}}Lu=f,
\]
where $\mathcal{A}$ is the family of linear operators $Lu={\rm trace}\left(AD^2u\right)$ for $A>0$   with eigenvalues $\lambda_j(A)$ that satisfy $\prod_{j=1}^{n} \lambda_j(A)=n^{-n}f^{n-1}$. Furthermore, if we take $nA$ equal to the matrix of cofactors of $D^2u$, then
\[
n^n\prod_{j=1}^{n} \lambda_j(A)=(\det D^2u)^{n-1}=f^{n-1},
\]
the infimum is realized 
 and the equation satisfied. Moreover, from the concavity of $\det^{1/n}(\cdot)$ any other choice of the eigenvalues would give a larger value than the prescribed $f$, making  $u$ a subsolution.

In other words, the Monge-Amp\`ere equation can be thought of as the infimum of a family of linear operators that consists of all affine transformations of determinant one of a given multiple  of the Laplacian.

\medskip

 \item[(b)] The fact that $\det D^2u$ can be represented as a concave fully nonlinear equation implies that pure second derivatives are subsolutions of an equation with bounded measurable coefficients and as such, are bounded from above.

Indeed, if we consider the second-order incremental quotient in the direction $e\in\partial B_1(0)$, 
 \[
 \delta(u,x_0,y)=u(x_0+he)+u(x_0-he)-2u(x_0)
 \]
and choose
\[
Lv={\rm trace}\left(AD^2v\right)
\]
with $nA$ the matrix of cofactors of  $D^2u(x_0)$, we have that $Lu(x_0)=f(x_0)$
while 
 on the other hand, the matrix
\[
B=\left[\frac{f(x_0+he)}{f(x_0)}\right]^\frac{n-1}{n}A
\]
satisfies
\[
\det B=n^{-n}f(x_0+he)^{n-1},
\]
which makes it eligible to compete for the minimum of
${\rm trace}\left(ND^2u(x_0+he)\right)$.
This implies, 
\[
\left[\frac{f(x_0+he)}{f(x_0)}\right]^\frac{n-1}{n}Lu(x_0+he)\geq f(x_0+he)
\]
or equivalently,
\[
Lu(x_0+he)\geq f(x_0+he)^\frac{1}{n}f(x_0)^\frac{n-1}{n}.
\]
We deduce that at a maximum of a second derivative $D_{ee}^2u$ the function $f$ must satisfy,
\[
f(x_0)^\frac{n-1}{n} D_{ee}^2\,f^{1/n}(x_0)\leq0.
\]
If $D_{ee}^2f$ is bounded and we have an appropriate barrier, plus control of the second derivatives of $u$ at the boundary of $\Omega$, we deduce that $u$ is not only convex but also semiconcave. For that purpose, the boundary and data must be smooth and the 
domain strictly convex. This allows for the construction of appropriate subsolutions as barriers.

 \medskip
 
 \item[(c)] Then, the last ingredient of the theory is that for a convex solution the equation $\prod_{j=1}^{n}\lambda_j=f$ with $f$ strictly positive implies that all $\lambda_j$ are strictly positive (and not merely non-negative). This implies that the operators involved with the minimization can be restricted to a uniformly elliptic family and the corresponding general theory applies.
In particular, Evans-Krylov theorem implies that solutions are $\mathcal{C}^{2,\alpha}$ and from there, as smooth as two derivatives better than $f$.

\end{enumerate}

\medskip

The discussion above suggests that one could carry out a similar program  for a non-local or fractional Monge-Amp\`ere equation of the form
\[
\inf L_Au=f
\]
where the set of operators $L_A$ corresponds to that of all affine transformations of determinant one of a given multiple of the fractional Laplacian. In fact, one may consider any concave function of the Hessian as in \cite{Ca-Ni-Sp2} as an infimum of affine transformations of the Laplacian, the affine transformations corresponding now to the different linearization coefficients of the function $F(\lambda_1,\ldots,\lambda_n)$ and consider the corresponding nonlocal operator.

One can take $\inf_{A\in\mathcal{A}} L_Au=f$, where 
$\mathcal{A}$ corresponds to a family of symmetric positive matrices with determinant bounded from above and below,
\[
0<\lambda\leq\det A\leq\Lambda,
\]
and 
\[
L_Au(x)=\int_{\mathbb{R}^n}\frac{u(x+y)+u(x-y)-2u(x)}{|A^{-1}y|^{n+2s}}\,dy.
\]

The kernel under consideration does not need to be necessarily $|A^{-1}x|^{-(n+2s)},$ it could be a more general kernel $K(Ax)$. In fact, the geometry of the domain is an important issue for the ``inherited from the boundary" regularity theory for degenerate operators depending on the eigenvalues of the Hessian, see \cite{Ca-Ni-Sp2}.

\medskip

In this article we shall set up a relatively simple framework of global solutions prescribing data at infinity and global barriers to avoid having to deal with the technical issues inherited from boundary data, which is rather complex for non-local equations. As in the second order case, we intend to prove:
\begin{itemize}
 \item[(a)] Existence of solutions.
 \item[(b)] Solutions are semiconcave, i.e. second derivatives are bounded from above.
 \item[(c)] Along each line, the fractional Laplacian is bounded from above and strictly positive.
 \item[(d)]  The operators that are close to the infimum remain strictly elliptic.
 \item[(e)] The non-local fully nonlinear theory developed in \cite{Caffarelli.Silvestre,Caffarelli.Silvestre2} applies, in particular the nonlocal Evans-Krylov theorem, and solutions are ``classical".
\end{itemize}

To be more precise, let us introduce  the non-local Monge-Amp\`ere operator  $\mathcal{D}_s$ that we are going to consider in the sequel, given by 
\begin{equation}\label{definicion.nonlocal.det}
\begin{split}
\mathcal{D}_s u(x)
&=\inf\bigg\{
\text{P.V.}\int_{\mathbb{R}^n}\frac{u(y)-u(x)}{|A^{-1}(y-x)|^{n+2s}}\,
\, dy\ \bigg|\ A>0,\ \det A=1\bigg\}\\
&=\inf\bigg\{\frac{1}{2}
\int_{\mathbb{R}^n}\frac{u(x+y)+u(x-y)-2u(x)}{|A^{-1}y|^{n+2s}}\, dy\ \bigg|\ A>0,\ \det A=1\bigg\}.
\end{split}
\end{equation}
We shall always use the definition that is most suitable to each case. Let us mention that if $u$ is convex, asymptotically linear, and $1/2<s<1$, then
\[
\lim_{s\to1}\big((1-s)\,\mathcal{D}_s u(x)\big)=\det(D^2u(x))^{1/n},
\]
up to a constant factor that depends only on the dimension $n$ (see Appendix A for a proof of this fact).

Another recent attempt to approach nonlocal Monge-Amp\`ere operators is the operator proposed in \cite{caffarelli.silvestre.MA}. The interested reader should also check \cite{Guillen.Schwab}.

\begin{remark}
 We can assume  without loss of generality that the matrices $A$ in the definition of 
$\mathcal{D}_s u(x)$  are symmetric and positive definite. This follows from the (unique) polar decomposition of $A^{-1}$, namely $A^{-1} = OS^{-1}$,
where $O$ is orthogonal and $S^{-1}$ is a positive definite symmetric matrix.

\end{remark}

We shall study the following Dirichlet problem, 
\begin{equation}\label{main.problem}
\left\{
\begin{split}
&\mathcal{D}_su(x)=u(x)-\phi(x) \qquad \text{in}
\
\mathbb{R}^n \\
&(u- \phi)(x)\to0\quad\!\text{as}\ |x|\to\infty,
\end{split}
\right.
\end{equation}
where  $1/2<s<1$ and we  prescribe boundary data at infinity $\phi(x)$ (that, at the same time,  acts as a smooth lower barrier). The results below can be extended to the problem
\begin{equation}\label{main.problem.g}
\left\{
\begin{split}
&\mathcal{D}_su(x)=g\big(x,u(x)\big)  \qquad \text{in}
\
\mathbb{R}^n \\
&(u- \phi)(x)\to0\quad\!\text{as}\ |x|\to\infty
\end{split}
\right.
\end{equation}
under appropriate assumptions on $g$ (see \cite{Jian.Wang} for a local analogue of problem \eqref{main.problem.g}). Let us now describe the precise hypothesis that we shall require on $g$ and $\phi$.

First,   $\phi\in\mathcal{C}^{2,\alpha}(\mathbb{R}^n)$ is strictly convex in compact sets and $\phi=\Gamma+\eta$ near infinity, with $\Gamma(x)$ a cone and 
\[
|\eta(x)|  \leq a |x|^{-\epsilon},
\qquad
|\nabla\eta(x)|  \leq a |x|^{-(1+\epsilon)},
\qquad\text{and}
\qquad
|D^2\eta(x) |\leq a |x|^{-(2+\epsilon)}
\]
for some  constants $a>0$ and $0<\epsilon<n$.
In particular, as $|x|\to\infty$,
\[
-(- \Delta)^s \eta(x)= O\big( |x| ^{-(2s+ \epsilon)}\big)
\]
(see Section \ref{sect.not.and.prelim} for the definition of the fractional Laplacian) and
\[
c_1|x|^{1-2s}\leq-(- \Delta)^s \Gamma(x)\leq c_2|x|^{1-2s}
\]
from homogeneity, where $c_1,c_2$ are some positive constants depending on the strict convexity of the section of $\Gamma$. We normalize $\phi$ so that $\phi(0)=0$, $\nabla\phi(0)=0$. 

\medskip

The model problem that we consider  is  $g(x,u(x))=u(x)- \phi(x)$. On the other hand, the general hypotheses on $g:
\mathbb{R}^{n+1}
\to\mathbb{R}$  that we shall consider are:
\begin{equation}\label{semiconvexity.condition.g}
\text{$g$ is globally semiconvex with constant $C$,}
\end{equation}
\begin{equation}\label{Lipschitz.condition.g}
\text{$x\mapsto g(x,t)$ is  Lipschitz continuous with constant $\textrm{Lip}(g)$, uniformly in $t$},
\end{equation}
and, there exists $\mu>0$ such that
\begin{equation}\label{monotonicity.condition.g}
g(x,t_1)-g(x,t_2)\geq 
\mu(t_1-t_2)\qquad\forall t_1,t_2\in\mathbb{R},\ t_1\geq t_2\quad \textrm{uniformly in}\ x.
\end{equation}

We would like to point out that hypothesis \eqref{semiconvexity.condition.g} implies that the function $g$ is  locally Lipschitz continuous (see for instance \cite[Proposition 2.1.7]{Cannarsa.Sinestrari}). In particular,
\[
\frac{|g(x,t)-g(y,t)|}{|x-y|} \leq\frac{2\,\textrm{osc}\big(g(\cdot,t),\overline B_{R/2}(x_0)\big)}{R}+CR \qquad \forall x,y\in B_{R/2}(x_0),
\]
for any $R>0$.
Therefore, hypothesis \eqref{Lipschitz.condition.g} could be replaced, for instance, by the following
\begin{equation}\label{Lipschitz.condition.g.2}
\begin{split}
x\mapsto g(x,t)\ &\text{is  Lipschitz continuous in $\mathbb{R}^{n}\setminus B_{R_0}(0)$ for some radius ${R_0}>0$} \\
&\text{with constant $\textrm{Lip}(g,\mathbb{R}^{n}\setminus B_{R_0}(0))$, uniformly in $t$}, 
\end{split}
\end{equation}
and
\begin{equation}
\textrm{osc}\big(g(\cdot,t),\overline B_{ R_0/2}(x_0)\big)\ \textrm{bounded in $t$}.
\end{equation}
In the sequel, we shall assume \eqref{Lipschitz.condition.g}  for simplicity.

\medskip

The paper is organized as follows. In Section \ref{sect.not.and.prelim} we present the notation, the notion of solution, and some preliminary results.
In Section \ref{sect4.unif.elipticity} we prove the main result of the paper,  
namely,  that matrices that are too degenerate do not count for the infimum in  \eqref{definicion.nonlocal.det}, effectively proving that the fractional Monge-Amp\`ere operator is locally uniformly elliptic and thus the known theory for uniformly elliptic nonlocal operators applies (see for instance \cite{Caffarelli.Silvestre} and the references therein).
 In
Section \ref{Sect.comparison} we prove a comparison principle for problem \eqref{main.problem.g}, and in  Section \ref{sect3.exist.lip} we prove Lipschitz continuity and semiconcavity of solutions to  problem \eqref{main.problem.g}.
Finally, in Section \ref{sect.existence.final.final} we prove existence of solutions to the model problem \eqref{main.problem}.

\medskip

\section{Notation and preliminaries}\label{sect.not.and.prelim}

In this section we are going to state notations and recall some basic results and definitions.

For square matrices, $A>0$ means positive definite and $A\geq0$ positive semidefinite. We shall denote $\lambda_i(A)$ the eigenvalues of $A$, in particular $\lambda_{\min}(A)$ and $\lambda_{\max}(A)$ are the smallest and largest eigenvalues, respectively.

We shall denote the $k$th-dimensional ball of radius 1 and center 0 by $B_1^{k}(0)=\{x\in\mathbb{R}^k:\ |x|\leq1\}$  and  the corresponding $(k-1)$th-dimensional sphere by $\partial B_1^{k}(0)=\{x\in\mathbb{R}^k:\ |x|=1\}$. Whenever $k$ is clear from context, we shall simply write $B_1(0)$ and $\partial B_1(0)$.  $\mathcal{H}^{k}$ stands for the $k$-dimensional Haussdorff measure.  We shall denote $\omega_{k}=
\mathcal{H}^{k-1}\big(\partial B_1^{k}(0)\big)=k\,|B_1^{k}(0)|=\frac{2\pi^{k/2}}{\Gamma(k/2)}$.

Given a function $u$, we shall denote the second-order increment  of $u$ at $x$ in the direction of $y$  as $\delta(u,x,y)=u(x+y)+u(x-y)-2u(x)$. 

Let $A\subset\mathbb{R}^n$ be an open set. We say that a function $u:A
\to
\mathbb{R}$ is semiconcave if it is continuous in $A$ and there exists $C
\geq 0$ such that
$\delta(u,x,y)
\leq C|y|^2$
for all $x, y\in 
\mathbb{R}^n$ such that $[x-y, x +y] 
\subset A$. The constant $C$  is called a
semiconcavity constant for $u$ in $A$.

Alternatively, a function $u$ is semiconcave in $A$ with constant $C$  if
 $u(x)- \frac{C}{2}|x|^2$ is concave in $A$. Geometrically, this means that the graph of $u$ can be touched from above at every point by a paraboloid of the type $a+\langle b,x\rangle+\frac{C}{2}|x|^2$.

A function $u$ is called semiconvex in $A$ if $-u$ is semiconcave.

Let us mention  here for the reader's convenience the definition of the fractional Laplacian ,
\[
-(- \Delta )^{s}u(x)=
c_{n,s}\,\text{P.V.}\int_{\mathbb{R}^n}\frac{u(y)-u(x)}{|y-x|^{n+2s}}\,
\, dy
=
\frac{c_{n,s}}{2}\,\int_{\mathbb{R}^n}\frac{u(x+y)+u(x-y)-2u(x)}{|y|^{n+2s}}\,
\, dy
\]
where $c_{n,s}$ is a normalization constant. Notice that  $-c_{n,s}^{-1}\,(- \Delta )^{s}u(x)$ belongs to the class of operators over which the infimum in the definition of $\mathcal{D}_s u(x)$ is taken.

\medskip

We recall from \cite{Caffarelli.Silvestre} the notion of viscosity solution that we are going to use in the sequel.
\begin{definition}\label{defn.viscosity.solution}
A function $u :\mathbb{R}^n \to \mathbb{R}$, upper (resp. lower) semicontinuous in $\overline \Omega$, is said to be a subsolution (supersolution) to $\mathcal{D}_su = f$, and we write $\mathcal{D}_su\geq f$ (resp. $\mathcal{D}_su \leq f$), if every time all the following happen,
\begin{itemize}
\item $x$ is a point in $\Omega$,
\item $N$ is an open neighborhood of $x$ in $\Omega$,
\item $\psi$ is some $C^2$ function in $\overline N$,
\item $\psi(x) = u(x)$,
\item $\psi(y) > u(y)$ (resp. $\psi(y) < u(y)$) for every $y \in N \setminus \{x\}$,
\end{itemize}
and if we let \[ v := \begin{cases}
               \psi &\text{in } N \\
	       u &\text{in } \mathbb{R}^n \setminus N \ ,
              \end{cases} \]
then we have $\mathcal{D}_sv(x) \geq f(x)$ (resp. $\mathcal{D}_sv(x) \leq f(x)$).
A solution is a function $u$ that is both a subsolution and a supersolution.
\end{definition}

The following lemma states that $\mathcal{D}_su$ can be evaluated classically  at those points $x$ where $u$ can be touched by a paraboloid. 

\begin{lemma}\label{viscosidad.clasico}
Let $1/2<s<1$ and $u:\mathbb{R}^n\to\mathbb{R}$ with asymptotically linear growth. If we have
$\mathcal{D}_{s}u\geq f$ in $\mathbb{R}^n$ (resp.  $\mathcal{D}_{s}u\leq f$)
 in the viscosity sense and $\psi$ is a $\mathcal{C}^2$ function that touches $u$ from above (below)  at a point $x$, then $\mathcal{D}_s u(x)$ is defined in the classical sense and $\mathcal{D}_s u(x)\geq f(x)$ (resp.  $\mathcal{D}_{s}u(x)\leq f(x)$).
\end{lemma}

\begin{proof}
Let us deal first with the subsolution case, that is, assume first that  $\psi\in\mathcal{C}^2$ touches $u$ from above at a point $x$.
Define for $r>0$,  
\[
 v_r(y) = \begin{cases}
        \psi(y) & \text{in } B_r(x) \\
        u(y) & \text{in } \mathbb{R}^n \setminus B_r(x).
       \end{cases}
\]
Then, we have that
$$
-c_{n,s}^{-1}\,(- \Delta )^{s}v_r(x)\geq \mathcal{D}_{s}v_r(x)\geq f(x)
$$
and then the arguments in the proof of \cite[Lemma 3.3]{Caffarelli.Silvestre} yield that $\delta(u,x,y)/|y|^{n+2s}$ is integrable. Therefore, $-(- \Delta)^su(x)$ is defined in the classical sense and $\mathcal{D}_s u(x)<+\infty$. Notice that,
\[
\lambda_{\min}(A)^{n+2s}\frac{\delta(u,x,y)}{|y|^{n+2s}}\leq\frac{\delta(u,x,y)}{|A^{-1}y|^{n+2s}}\leq
\lambda_{\max}(A)^{n+2s}\frac{\delta(u,x,y)}{|y|^{n+2s}}.
\]
Thus,  $\delta(u,x,y)/|A^{-1}y|^{n+2s}$ is integrable and 
\begin{equation}\label{defn.L_A}
 L_A u(x)=
\frac{1}{2}\int_{\mathbb{R}^n}\frac{\delta(u,x,y)}{|A^{-1}y|^{n+2s}}\, dy
\end{equation}
is also defined in the classical sense. By definition of viscosity solution, we have
\[
L_A u(x)+ L_A (v_r-u)(x)\geq \mathcal{D}_s v_r(x)\geq f(x).
\]
But then, $0\leq\delta(v_r-u,x,y)\leq\delta(v_{r_0}-u,x,y)$ for all $r<r_0$, $\delta(v_{r_0}-u,x,y)/|A^{-1}y|^{n+2s}$ is integrable and $\delta(v_r-u,x,y)\to0$ as $r\to0$. Hence, by the dominated convergence theorem, $L_A (v_r-u)(x)\to0$ as $r\to0$. We conclude
$L_A u(x)\geq f(x)$ in the classical sense. Since the matrix $A$ is arbitrary and we could pick any matrix $A>0$ with $\det A=1$, we have that $\mathcal{D}_s u(x)\geq f(x)$ in the classical sense.

\medskip
In the supersolution case, that is, when  $\psi\in\mathcal{C}^2$ touches $u$ from below at  $x$, some modifications are required. Fix $\epsilon>0$, arbitrary, and let $A_\epsilon>0$ with $\det A_\epsilon=1$ such that 
\[
L_{A_\epsilon} v_r(x)\leq f(x)+\epsilon.
\]
It is easy to see that $\delta(v_r,x,y)$ is non-decreasing in $r$ and $\delta(v_r,x,y)\to\delta(u,x,y)$ as $r\to0$. By the monotone convergence theorem, $\delta(u,x,y)/|A_{\epsilon}^{-1}y|^{n+2s}$ is integrable and $L_{A_\epsilon} u(x)\leq f(x)+\epsilon$ in the classical sense.
We find that
\[
\mathcal{D}_s u(x)\leq L_{A_\epsilon} u(x)\leq f(x)+\epsilon,
\]
and we conclude letting $\epsilon\to0$, since it is arbitrary.
\end{proof}

\medskip

\section{Local uniform ellipticity of the fractional Monge-Amp\`ere equation}\label{sect4.unif.elipticity}

In this section we shall prove that the infimum in the definition of $\mathcal{D}_s$, see \eqref{definicion.nonlocal.det}, cannot be realized by matrices that are too degenerate, effectively proving that the fractional Monge-Amp\`ere operator is locally uniformly elliptic. Then, existing theory for uniformly elliptic operators is available (see \cite{Caffarelli.Silvestre, Caffarelli.Silvestre2} and the references therein).

To this aim, consider the following approximating, non-degenerate operator, 
\begin{equation}\label{definicion.nonlocal.det.approximate}
\begin{split}
\mathcal{D}_s^\theta u(x)&=\inf\bigg\{
\text{P.V.}\int_{\mathbb{R}^n}\frac{u(y)-u(x)}{|A^{-1}(y-x)|^{n+2s}}\,
\, dy\ \bigg|\ A>0,\ \det A=1,\  \lambda_{\min}(A)\geq\theta\bigg\}\\
&=\inf\bigg\{\frac{1}{2}
\int_{\mathbb{R}^n}\frac{u(x+y)+u(x-y)-2u(x)}{|A^{-1}y|^{n+2s}}\, dy\ \bigg|\ A>0,\ \det A=1,\  \lambda_{\min}(A)\geq\theta\bigg\}.
\end{split}
\end{equation}
Let us point out that the conditions $\det A=1$, and $\lambda_{\min}(A)\geq\theta$ imply $\lambda_{\max}(A)\leq\theta^{1-n}$ and this bound is realized by matrices with eigenvalues $\theta$ (simple) and $\theta^{1-n}$ (multiplicity $n-1$). Therefore, $\mathcal{D}_s^\theta$ belongs to the class of uniformly elliptic, nonlocal operators with extremal Pucci operators 
\[
\mathcal{M}_{\theta,\theta^{1-n}}^{+} u(x)=\sup\bigg\{\frac{1}{2}
\int_{\mathbb{R}^n}\frac{u(x+y)+u(x-y)-2u(x)}{|A^{-1}y|^{n+2s}}\, dy\ \bigg|\ \theta\, I\leq A \leq \theta^{1-n}I\bigg\}
\]
and
\[
\mathcal{M}_{\theta,\theta^{1-n}}^{-} u(x)=\inf\bigg\{\frac{1}{2}
\int_{\mathbb{R}^n}\frac{u(x+y)+u(x-y)-2u(x)}{|A^{-1}y|^{n+2s}}\, dy\ \bigg|\ \theta\, I\leq A \leq \theta^{1-n}I\bigg\}.
\]
Observe that in general $\mathcal{M}_{\theta,\theta^{1-n}}^{-} u(x)<
\mathcal{D}_s^\theta u(x),$ as the class of matrices over which the infimum is taken is broader for the Pucci operator. 

\medskip

The main result of this section and of the paper is the following.

\begin{theorem}\label{main.result.nondegeneracy.monge.ampere}
Consider $\frac12<s<1$ and let $u$ be Lipschitz continuous and semiconcave (with constants $L$ and $C$ respectively) and such that\
\begin{equation}\label{problem.main.estimate}
(1-s)\,\mathcal{D}_su(x)\geq \eta_0 \quad \forall x
\in
\Omega
\end{equation}
in the viscosity sense for some constant $\eta_0>0$ and $\Omega\subset\mathbb{R}^n$. Then, 
\begin{equation}\label{eq.main.result}
\mathcal{D}_s  u(x)=\mathcal{D}_s^\theta  u(x) \quad \forall x
\in
\Omega
\end{equation}
in the classical sense,
for $\mathcal{D}_s^\theta$  the approximating operator defined by \eqref{definicion.nonlocal.det.approximate} and
\[
\theta<
\left(
\frac{\mu_0}{n\mu_1}
\right)^
\frac{n-1}{2s}
\]
with  $\mu_0,\mu_1$ defined in \eqref{def.mu0} and \eqref{def.mu1} below.
\end{theorem}

\begin{remark}[Limits as $s\to1$] 
It can be checked that
 $$
\frac{\mu_0}{\mu_1}
= O\big(\eta_0^n (2s-1)^{n}\big)
$$
as $s\to1$. In particular, Theorem \ref{main.result.nondegeneracy.monge.ampere} is stable in the limit as $s\to1$. 
\end{remark}

\medskip

It is illustrative for the sequel  to show how the ideas in the proof of Theorem \ref{main.result.nondegeneracy.monge.ampere} work in the local case. 
More precisely, assume  $u$ semiconcave with constant $C$ and such that 
\[
\frac{\omega_n}{4n}\cdot\inf\big\{{\rm trace}\big(AA^tD^2u(x)\big)\, |\ \det A=1\big\}\geq\eta_0\qquad\forall x\in\Omega
\]
(about the normalization  $(4n)^{-1}\omega_n$, recall \eqref{problem.main.estimate} and Lemma \ref{caract.determ.2}).
We want to prove that the Monge-Amp\`ere operator is actually non-degenerate, that is, 
\[
\inf\big\{{\rm trace}\big(AA^tD^2u(x)\big)\, |\ \det A=1\big\}=\inf\big\{{\rm trace}\big(AA^tD^2u(x)\big)\, |\ \det A=1,\ \lambda_{\min}(A)\geq\theta\big\}
\]
for some $\theta>0$.
The proof has two steps:

\noindent\quad{}1. The second derivative of $u$ in the direction $e$ is strictly positive and bounded (uniformly) for every direction. More precisely,
\begin{equation}\label{heuristic.local.first.part}
0<\bar\mu_0\leq u_{ee}(x)\leq C\qquad\forall e\in\partial B_1(0).
\end{equation}
for $\bar\mu_0$ independent of $e$ (given by \eqref{expres.bar.mu0} below), and $C$ the semiconcavity constant of $u$.
The proof of the upper bound follows from the definition of semiconcavity. For the lower bound, choose $A=PJP^t$ with $J$ a diagonal matrix with eigenvalues $\epsilon$ (single) and $\epsilon^\frac{1}{1-n}$ (multiplicity $n-1$), and $P$ an orthogonal matrix whose $i$-th column is $e$ (notice that $\det(A) =1$).
Then,
\[
\begin{split}
\frac{4n\,\eta_0}{\omega_n}\leq\textrm{trace}(AA^tD^2u(x))&=\sum_{j=1}^{n}\lambda_j^2(A)(P^tD^2u(x)P)_{jj}\\
&=\epsilon^2\,(P^tD^2u(x)P)_{ii}+\epsilon^\frac{2}{1-n}\sum_{j\neq i}^{n}(P^tD^2u(x)P)_{jj}\\
&\leq \epsilon^2\,(P^tD^2u(x)P)_{ii}+C(n-1)\epsilon^\frac{2}{1-n}
\end{split}
\]
by semiconcavity. Choosing $\epsilon$ small enough, e.g. $\epsilon=\big(\frac{1}{2}C(n-1)\omega_nn^{-1}\eta_0^{-1}\big)^\frac{n-1}{2}$ we get 
\[
0<\bar\mu_0\leq\big(P^tD^2u(x)P\big)_{ii}\leq C\qquad \forall i=1,\ldots,n,
\]
which is equivalent to \eqref{heuristic.local.first.part}. For future reference, $\bar\mu_0$ is given by
\begin{equation}\label{expres.bar.mu0}
\bar\mu_0=\left(\frac{2n\eta_0}{\omega_n}\right)^n(C(n-1))^{1-n}.
\end{equation}

\noindent\quad{}2. The infimum in the  Monge-Amp\`ere operator cannot be achieved for matrices that are too degenerate. More precisely, let $A$ with $\det(A)=1$ and write $A=PJP^t$ with $P$ orthogonal; then
\begin{equation}\label{heuristic.local.second.part}
\textrm{trace}(AA^tD^2u(x))=\sum_{i=1}^{n}\lambda_i^2(A)(P^tD^2u(x)P)_{ii}\geq\bar\mu_0\sum_{i=1}^{n}\lambda_i^2(A)\geq\bar\mu_0\,
\lambda_{\min}(A)^{- \frac{2}{n-1}},
\end{equation}
using that  $1=\det(A)\leq\lambda_{\min}(A)\lambda_{\max}(A)^{n-1}$. We conclude that matrices with very small eigenvalues will produce very large operators that will not count for the infimum (see the proof of Theorem \ref{main.result.nondegeneracy.monge.ampere} for details).

\bigskip

For simplicity, we shall assume that $0 \in\Omega$ and then prove \eqref{eq.main.result} for $x=0$. Note for the sequel that since $u$ is semiconcave, Lemma \ref{viscosidad.clasico}  implies that $\mathcal{D}_s  u(x)$ is defined in the classical sense for all $x\in\Omega$ and \eqref{problem.main.estimate} holds pointwise.

The proof of Theorem \ref{main.result.nondegeneracy.monge.ampere} has,  again, two parts. In the first part we prove that the (one-dimensional) fractional laplacian of the restriction of $u$ to any line is positive and bounded from above. Then, in the second part, we shall use this fact to prove that 
\begin{equation}\label{goal.second.part.nonlocal.heuris}
(1-s)\int_{\mathbb{R}^n}\frac{u(y)-u(0)}{|A^{-1}y|^{n+2s}}\,dy\geq \frac{ \mu_0\,\omega_n}{2n}\, \sum_{j=1}^{n}\lambda_j^{2s}(A)\geq \frac{ \mu_0\,\omega_n}{2n}\, \lambda_{\min}(A)^{- \frac{2s}{n-1}}.
\end{equation}
for $\mu_0$ given by \eqref{def.mu0}. Therefore the infimum in the fractional Monge-Amp\`ere operator cannot be achieved for matrices that are too degenerate.

The two parts we have mentioned correspond to the following two results.

\begin{proposition}\label{first.part.main}
Assume the same hypotheses of Theorem \ref{main.result.nondegeneracy.monge.ampere}. Then, for every $e\in\partial B_1(0)$, 
\[
0<\mu_0\leq- (1-s)\big(- \Delta \big)_{e}^su(0)=(1-s)\int_{\mathbb{R}} \frac{u(te)-u(0)}{|t|^{1+2s}}\,dt\leq\mu_1,
\]
with
\begin{equation}\label{def.mu0}
\mu_0=C_1^{1-n}C_2^{-1}\left(\frac{\eta_0}{2}\right)^n 
\end{equation}
for $C_1,C_2$ defined in \eqref{def.C1} and \eqref{def.C2}, and 
\begin{equation}\label{def.mu1}
\mu_1=\frac{1-s}{2}\int_{\mathbb{R}} \frac{\min\{2L\,|t|,C|t|^2\}}{|t|^{1+2s}}\,dt=\frac{L^{2-2s}}{2s-1}\left(\frac{C}{2}\right)^{2s-1}.
\end{equation}
\end{proposition}

\begin{remark}\label{remark.limit.first.part}
Proposition \ref{first.part.main} yields \eqref{heuristic.local.first.part} in the limit as $s\to1$ since $\lim_{s\to1}\mu_0=\bar\mu_0/2$ (with $\bar\mu_0$ defined by \eqref{expres.bar.mu0}), $\lim_{s\to1}\mu_1=C/2$ and
\[
\lim_{s\to1}(1-s)\int_{\mathbb{R}} \frac{u(te)-u(0)}{|t|^{1+2s}}\,dt=\frac{u_{ee}(0)}{2}.
\]
\end{remark}

\begin{proposition}\label{second.part.main}
Assume $\epsilon_1,\ldots,\epsilon_n$ are positive constants such that $\prod_{j=1}^n\epsilon_j=1$. Then, in the same hypotheses of Theorem \ref{main.result.nondegeneracy.monge.ampere}, we have,
\[
(1-s)\int_{\mathbb{R}^n}\frac{u(y)-u(0)}{\left(\sum_{j=1}^n\epsilon_j^2y_j^2\right)^\frac{n+2s}{2}}\,dy
 \geq
\frac{ \mu_0\,\omega_n}{2n}\cdot\sum_{j=1}^{n}\frac{1}{\epsilon_j^{2s}},
\]
with  $\mu_0$ defined in \eqref{def.mu0}.
\end{proposition}

\begin{remark}
Proposition \ref{second.part.main} implies \eqref{goal.second.part.nonlocal.heuris},
which yields \eqref{heuristic.local.second.part} in the limit as $s\to1$ since
\[
\lim_{s\to1}(1-s)\int_{\mathbb{R}^n}\frac{u(y)-u(0)}{|A^{-1}y|^{n+2s}}\,dy=\frac{\omega_n}{4n}\,\textrm{trace}(AA^tD^2u(x)).
\]
\end{remark}

\medskip

Propositions \ref{first.part.main} and \ref{second.part.main} (that we prove below) allow to prove the main result of this section, Theorem \ref{main.result.nondegeneracy.monge.ampere}.
\begin{proof}[Proof of Theorem \ref{main.result.nondegeneracy.monge.ampere}.]
Consider a symmetric matrix $A>0$ with  $\det A=1$ and $\lambda_{\min}(A)<\frac{1}{k}$. We can write $A=PJP^t$, and denote $\tilde{u}(y)=u(Py)$. Observe that then Proposition \ref{second.part.main} (see also \eqref{goal.second.part.nonlocal.heuris}) implies
\[
\int_{\mathbb{R}^n}\frac{u(y)-u(0)}{|A^{-1}y|^{n+2s}}\,dy
=
\int_{\mathbb{R}^n}\frac{u(Py)-u(0)}{|J^{-1}y|^{n+2s}}\,dy
=
\int_{\mathbb{R}^n}\frac{\tilde u(y)- \tilde u(0)}{\left(\sum_{j=1}^n\epsilon_j^2y_j^2\right)^\frac{n+2s}{2}}dy
>
\frac{\mu_0\,\omega_n}{2n(1-s)} \,k^\frac{2s}{n-1}
\]
and we get the estimate
\begin{equation}\label{estimate1.main.proof}
 \inf\left\{
\int_{\mathbb{R}^n}\frac{u(y)-u(0)}{|A^{-1}y|^{n+2s}}\,dy\ \bigg|\ A>0,\ \det A=1,
\ \lambda_{\min}(A)<\frac{1}{k}
\right\}\geq
\frac{\mu_0\,\omega_n}{2n(1-s)} \,k^\frac{2s}{n-1}.
\end{equation}
Observe that by
 choosing $A=I$, Proposition \ref{first.part.main} yields
\begin{multline}\label{estimate2.main.proof}
\inf\left\{
\int_{\mathbb{R}^n}\frac{u(y)-u(0)}{|A^{-1}y|^{n+2s}}\,dy\ \bigg|\ A>0,\ \det A=1
\right\}
\leq
\int_{\mathbb{R}^n}\frac{u(y)-u(0)}{|y|^{n+2s}}\,dy\\
=
\int_{\partial{B_1(0)}}\int_0^\infty\frac{u(re)-u(0)}{r^{1+2s}}\,dr\,d\mathcal{H}^{n-1}(e)
\leq
\frac{\mu_1
\omega_{n}}{2(1-s)}.
\end{multline}
Therefore, from \eqref{estimate1.main.proof} and \eqref{estimate2.main.proof} we have that whenever
$k>
\left(
n\mu_1\mu_0^{-1}
\right)^
\frac{n-1}{2s}$, 
\begin{multline*}
 \inf\left\{
\int_{\mathbb{R}^n}\frac{u(y)-u(0)}{|A^{-1}y|^{n+2s}}\,dy\ \bigg|\ A>0,\ \det A=1,
\ \lambda_{\min}(A)<\frac{1}{k}
\right\}\\
>
\inf\left\{
\int_{\mathbb{R}^n}\frac{u(y)-u(0)}{|A^{-1}y|^{n+2s}}\,dy\ \bigg|\ A>0,\ \det A=1
\right\}.
\end{multline*}
This implies \eqref{eq.main.result}, since
\begin{multline*}
\inf\left\{
\int_{\mathbb{R}^n}\frac{u(y)-u(0)}{|A^{-1}y|^{n+2s}}\,dy\ \bigg|\ A>0,\ \det A=1
\right\}
\\
=
{\rm min}\Bigg\{
\inf\left\{\int_{\mathbb{R}^n}\frac{u(y)-u(0)}{|A^{-1}y|^{n+2s}}\,dy\ \bigg|\ A>0,\ \det A=1,
\ \lambda_{\min}(A)<\frac{1}{k}
\right\},
\\
\inf\left\{\int_{\mathbb{R}^n}\frac{u(y)-u(0)}{|A^{-1}y|^{n+2s}}\,dy\ \bigg|\ A>0,\ \det A=1,
\ \lambda_{\min}(A)\geq\frac{1}{k}
\right\}\Bigg\}.
\end{multline*}
\end{proof}

The rest of this section is devoted to the proof of Propositions \ref{first.part.main} and \ref{second.part.main}.

\subsection{Proof of Propositions \ref{first.part.main} and \ref{second.part.main}}
Our goal is to prove that the (one-dimensional) fractional Laplacian of the restriction of $u$ to any line is positive and bounded from above. In the proof of Proposition \ref{first.part.main} we need several partial results.

In the sequel, we denote $\bar y=(y_2,\ldots,y_n)\in\mathbb{R}^{n-1}$ and $v(y)=u(y)-u(0)$.

\begin{lemma}\label{first.part.lemma1}
Let $\epsilon>0$ and assume the same hypotheses of Theorem \ref{main.result.nondegeneracy.monge.ampere}. Then,
\[
(1-s)\int_{\mathbb{R}^n}\frac{u(y_1,\bar y)-u(y_1,\bar0)}{\left(\epsilon^2y_1^2+\epsilon^\frac{-2}{n-1}|\bar y|^2\right)^\frac{n+2s}{2}}\,dy\leq C_1\cdot\epsilon^\frac{2s}{n-1}
\]
with
\begin{equation}\label{def.C1}
 C_1=\frac{\sqrt{\pi}\cdot\Gamma\left(\frac{n-1}{2}+s\right)}{\Gamma\left(\frac{n}{2}+s\right)}\cdot\frac{\mu_1\omega_{n-1}}{2},
\end{equation}
for $\mu_1$ given by \eqref{def.mu1}.
\end{lemma}
\begin{proof}
Since $u$ is Lipschitz and semiconcave, we have
\[
\int_{\mathbb{R}^n}\frac{u(y_1,\bar y)-u(y_1,\bar0)}{\left(\epsilon^2y_1^2+\epsilon^\frac{-2}{n-1}|\bar y|^2\right)^\frac{n+2s}{2}}\,dy
\leq
\frac{1}{2}\int_{\mathbb{R}^n}\frac{\min\left\{2L\,|\bar y|,C\,|\bar y|^2\right\}}{\left(\epsilon^2y_1^2+\epsilon^\frac{-2}{n-1}|\bar y|^2\right)^\frac{n+2s}{2}}\,dy.
\]
A change of variables
\[
z_1=\epsilon^{\frac{n}{n-1}}\, y_1\,|\bar y|^{-1},\qquad z_j=y_j,\quad j=2,\ldots,n,
\]
yields,
\[
\int_{\mathbb{R}^n}\frac{\min\left\{2L\,|\bar y|,C\,|\bar y|^2\right\}}{\left(\epsilon^2y_1^2+\epsilon^\frac{-2}{n-1}|\bar y|^2\right)^\frac{n+2s}{2}}\,dy=
\epsilon^\frac{2s}{n-1}\cdot
\int_{\mathbb{R}}
\left(1+z_1^2\right)^{- \frac{n+2s}{2}}dz_1\cdot
\int_{\mathbb{R}^{n-1}}\frac{\min\left\{2L\,|\bar z|,C\,|\bar z|^2\right\}}{|\bar z|^{n-1+2s}}\,d\bar{z},
\]
and  the result follows noticing that both integrals on the right-hand side are constant.
\end{proof}

\begin{lemma}\label{first.part.lemma2}
We have,
\[
\int_{\mathbb{R}^n}\frac{v(y_1,\bar0)}{\left(\epsilon^2y_1^2+\epsilon^\frac{-2}{n-1}|\bar y|^2\right)^\frac{n+2s}{2}}\,dy
=
C_2\ \epsilon^{-2s}\int_{\mathbb{R}} \frac{v(y_1,\bar0)}{|y_1|^{1+2s}}\,dy_1
\]
where
\begin{equation}\label{def.C2}
C_2=\omega_{n-1}\cdot\frac{\Gamma\left(\frac{n-1}{2}\right)\Gamma\left(s+\frac12\right)}{2\,\Gamma\left(\frac{n}{2}+s\right)}.
\end{equation}
\end{lemma}

\begin{proof}
A change of variables $z_1=y_1$,
$z_j=\frac{y_{j}}{\epsilon^{\frac{n}{n-1}}\,y_{1}}$, $j=2,\ldots,n$
yields,
\[
\begin{split}
\int_{\mathbb{R}^n}\frac{v(y_1,\bar0)}{\left(\epsilon^2y_1^2+\epsilon^\frac{-2}{n-1}|\bar y|^2\right)^\frac{n+2s}{2}}\,dy
=
 \frac{1}{\epsilon^{n+2s}}\int_{\mathbb{R}^{n-1}}\int_{\mathbb{R}} \frac{v(y_1,\bar0)}{|y_1|^{n+2s}}\,\left(1+\epsilon^{- \frac{2n}{n-1}}\frac{|\bar y|^2}{y_1^2}\right)^{- \frac{n+2s}{2}}\,d\bar{y}\,dy_1\\
 =
 \frac{1}{\epsilon ^{2s}}\int_{\mathbb{R}} \frac{v(z_1,\bar0)}{|z_1|^{1+2s}}\,dz_1\,\int_{\mathbb{R}^{n-1}}\frac{d\bar{z}}{\left(1+|\bar{z}|^2\right)^{\frac{n+2s}{2}}}= \frac{C_2}{\epsilon ^{2s}}\int_{\mathbb{R}} \frac{v(y_1,\bar0)}{|y_1|^{1+2s}}\,dy_1.\qedhere
 \end{split}
\]
\end{proof}

Lemmas \ref{first.part.lemma1} and \ref{first.part.lemma2} allow us to prove that the one-dimensional fractional laplacian of the restriction $v(y_1,\bar0)$ is strictly positive.

\begin{lemma}\label{first.part.lemma3}
Under the same hypotheses of Theorem \ref{main.result.nondegeneracy.monge.ampere}, we have
\[
(1-s)\int_{\mathbb{R}} \frac{u(y_1,\bar0)-u(0)}{|y_1|^{1+2s}}\,dy_1\geq\mu_0,
\]
where $\mu_0$ is given by \eqref{def.mu0}.
\end{lemma}

\begin{proof}
From Lemmas \ref{first.part.lemma1} and \ref{first.part.lemma2}, we have that
\[
\frac{C_1\cdot\epsilon^\frac{2s}{n-1}}{1-s}\geq
\int_{\mathbb{R}^n}\frac{v(y_1,\bar y)}{\left(\epsilon^2y_1^2+\epsilon^\frac{-2}{n-1}|\bar y|^2\right)^\frac{n+2s}{2}}\,dy-C_2\ \epsilon^{-2s}\int_{\mathbb{R}} \frac{v(y_1,\bar0)}{|y_1|^{1+2s}}\,dy_1.
\]
Then, by \eqref{problem.main.estimate} and the definition of $\mathcal{D}_s$ we get
\[
\int_{\mathbb{R}^n}\frac{v(y_1,\bar y)}{\left(\epsilon^2y_1^2+\epsilon^\frac{-2}{n-1}|\bar y|^2\right)^\frac{n+2s}{2}}\,dy\geq
\inf\left\{
\int_{\mathbb{R}^n}\frac{u(y)-u(0)}{|A^{-1}y|^{n+2s}}\,dy\ \bigg|\ A>0,\ \det A=1
\right\}
\geq \frac{\eta_0}{1-s}>0.
\]
Therefore,
\[
C_1\cdot\epsilon^\frac{2s}{n-1}\geq
\eta_0-C_2\ \epsilon^{-2s}(1-s)\int_{\mathbb{R}} \frac{v(y_1,\bar0)}{|y_1|^{1+2s}}\,dy_1.
\]
We get the result from this expression by choosing
$
\epsilon=\left(\frac{\eta_0}{2C_1}\right)^\frac{n-1}{2s}.
$
\end{proof}

From Lemma \ref{first.part.lemma3} we can finally prove Proposition \ref{first.part.main}.

\begin{proof}[Proof of Proposition \ref{first.part.main}.]
First, we are going to prove that the one-dimensional fractional Laplacian of the restriction of $u$ to any line is bounded above. Indeed, from the Lipschitz continuity and semiconcavity of $u$, 
\[
\int_{\mathbb{R}} \frac{u(te)-u(0)}{|t|^{1+2s}}\,dt=
\int_{\mathbb{R}} \frac{\frac12u(te)+\frac12u(-te)-u(0)}{|t|^{1+2s}}\,dt
\leq
\frac12\int_{\mathbb{R}} \frac{\min\{2L\,|t|,C|t|^2\}}{|t|^{1+2s}}\,dt=\frac{\mu_1}{1-s},
\]
where $\mu_1$ is given by \eqref{def.mu1}.

Now,  fix $e\in\partial B_1(0)$, and choose $P$ such that $e$ is its first column and the rest of columns complete an orthonormal basis of $\mathbb{R}^n$. Notice that $\tilde u(x)=u(Px)$ is in the hypotheses of Theorem \ref{main.result.nondegeneracy.monge.ampere}. Hence, we
 can  apply Lemma \ref{first.part.lemma3} to  $\tilde u$ and get
\[
(1-s)\int_{\mathbb{R}} \frac{\tilde u(y_1,\bar0)- \tilde u(0)}{|y_1|^{1+2s}}\,dy_1\geq\mu_0,
\]
but then, $\tilde u(y_1,\bar0)=\tilde u(y_1e_1)=u(y_1Pe_1)=u(y_1e)$ by definition of $P$.
\end{proof}

Next, we provide the proof of Proposition \ref{second.part.main} that uses Proposition \ref{first.part.main}.

\begin{proof}[Proof of Proposition \ref{second.part.main}.]
Our aim is to prove that the infimum in the fractional Monge-Amp\`ere operator is not realized by matrices that are very degenerate. From Proposition \ref{first.part.main}, we have
\[
\begin{split}
\int_{\mathbb{R}^n}\frac{u(y)-u(0)}{\left(\sum_{j=1}^n\epsilon_j^2y_j^2\right)^\frac{n+2s}{2}}\,dy
&=
\int_{\partial{B_1(0)}}\int_0^\infty\frac{u(re)-u(0)}{r^{1+2s}}\,dr\;\frac{1}{\left(\sum_{j=1}^n\epsilon_j^2e_j^2\right)^\frac{n+2s}{2}}\,d\mathcal{H}^{n-1}(e)\\
& \geq
\frac{ \mu_0}{2(1-s)}
\int_{\partial{B_1(0)}}\frac{1}{\left(\sum_{j=1}^n\epsilon_j^2e_j^2\right)^\frac{n+2s}{2}}\,d\mathcal{H}^{n-1}(e).
\end{split}
\]
Proposition \ref{estimate.kernel.on.sphere} yields the estimate,
\[
\int_{\partial{B_1(0)}}\frac{1}{\left(\sum_{j=1}^n\epsilon_j^2e_j^2\right)^\frac{n+2s}{2}}\,d\mathcal{H}^{n-1}(e)
\geq
\frac{\omega_n}{n}\sum_{j=1}^{n}\frac{1}{\epsilon_j^{2s}},
\]
where we have used that $\prod_{j=1}^n\epsilon_j=1$. This completes the proof.
\end{proof}

\medskip

\section{Comparison and Uniqueness}\label{Sect.comparison}

Next, we prove a comparison principle that yields uniqueness for  problem \eqref{main.problem.g}.
Notice that the same arguments apply to the operator $(1-s)\mathcal{D}_s$ giving a stable result in the limit as $s\to1$. 

\begin{theorem} \label{them.comparison}
Assume $1/2<s<1$, and
 let $g:\mathbb{R}^{n+1}\to\mathbb{R}$ a continuous function satisfying \eqref{monotonicity.condition.g}. Consider $\phi\in\mathcal{C}^{2,\alpha}(\mathbb{R}^n)$, and $u\in USC$ and $v\in LSC$ such that
\[
 \left\{
\begin{split}
&\mathcal{D}_s u(x) \geq g(x,u) \qquad \text{in}
\quad
\mathbb{R}^n \\
&(u- \phi)(x)\to0\quad\!\text{as}\ |x|\to\infty
\end{split}
\right.
\qquad\text{and}
\qquad
 \left\{
\begin{split}
&\mathcal{D}_s v(x)\leq g(x,v) \qquad \text{in}
\quad
\mathbb{R}^n \\
&(v- \phi)(x)\to0\quad\!\text{as}\ |x|\to\infty. 
\end{split}
\right.
\]
in the viscosity sense.  Then, $u\leq v$ in $\mathbb{R}^n$.
\end{theorem}

\begin{remark}
It is also possible to assume 
 $t\mapsto g(x,t)$ strictly increasing for any $x\in\mathbb{R}^n$ instead of \eqref{monotonicity.condition.g} to derive a contradiction in \eqref{final.knot}. 
\end{remark}

\begin{proof}
Let us first present the ideas of the proof in the case when $u,v$ are a classical sub- and supersolution, then we shall consider the  viscosity counterparts.

Since we seek to prove $u \leq v$, let us assume to the contrary that  $\sup_{\mathbb{R}^n}(u-v)>0$. As $(u- v)(x)\to0$ as $|x|\to\infty$,  there exists $x_0 \in\mathbb{R}^n$ such that
 \[
 (u- v)(x_0) = \sup_{\mathbb{R}^n}(u-v)>0.
 \]
 Fix $\delta>0$, arbitrary, and let $A_\delta>0$ with $\det A_\delta=1$, such that 
\[
L_{A_\delta} v(x_0)\leq \mathcal{D}_s v(x_0) +\delta\leq g(x_0,v(x_0))+\delta,
\]
for $L_{A_\delta}$  defined as in \eqref{defn.L_A}.
 On the other hand, for the same matrix,
\[
L_{A_\delta} u(x_0)\geq\mathcal{D}_s u(x_0) \geq  g(x_0,u(x_0)).
\]
At a maximum point $\delta(u-v,x_0,y)\leq0$, and
\[
0  \geq L_{A_\delta} (u-v)(x_0)\geq g(x_0,u(x_0))-g(x_0,v(x_0))-   \delta.
\]
Therefore,
since $\delta$ is arbitrary, we can let $\delta\to0$ and get 
\[
g(x_0,v(x_0)) \geq g(x_0,u(x_0)),
\]
a contradiction with the fact that $g(x_0,\cdot)$ is strictly increasing.

\medskip

In the general case, we cannot be certain that $L_{A_\delta} u(x_0)$ and $L_{A_\delta} v(x_0)$ above are well defined, since $u$ and $v$ may not have the necessary regularity. To remedy that we shall use sup- and inf-convolutions and  work with regularized functions. However, we shall rather apply the regularizations to the functions $\bar u=u- \phi$ and $\bar v=v- \phi$, since they are  bounded above and  below respectively (notice that $\bar u\in USC$, $\bar v\in LSC$, and $\bar u(x),\bar v(x) \to0$ as $|x|\to\infty$ imply that $\bar u,\bar v$ have respectively a maximum and a minimum). 

Consider the sup- and inf-convolution of $\bar u, \bar v$, respectively,
\begin{equation}\label{sup.convolution}
\bar u^\epsilon(x)=\sup_{y}
\left\{ \bar u(y)- \frac{|x-y|^2}{\epsilon}\right\} 
\end{equation}
and 
\[
\bar v_\epsilon(x)=\inf_{y}
\left\{ \bar v(y)+ \frac{|x-y|^2}{\epsilon}\right\}.
\]
Before we start with the proof, let us recall for the reader's convenience two properties of $\bar u^\epsilon$ that we shall use in the sequel. Analogous properties hold for  $\bar v_\epsilon$ noticing that $\bar v_\epsilon=-(- \bar v)^\epsilon$.
\begin{enumerate}
 \item $\bar u^\epsilon$ is bounded above. Since $\bar u$   is bounded above by some constant $C$, we have
 \[
 \bar u^\epsilon(x)\leq\sup_{y}
\left\{ C- \frac{|x-y|^2}{\epsilon}\right\} =C.
 \]
 \item The supremum in the definition of \eqref{sup.convolution} is achieved. In fact,
\begin{equation}\label{supconvol.formula2}
 \bar u^\epsilon(x)=\sup_{|y-x|^2\leq2\| \bar u\|_\infty\epsilon}
\left\{ \bar u(y)- \frac{|x-y|^2}{\epsilon}\right\}=\bar u(x^*)- \frac{|x-x^*|^2}{\epsilon}
 \end{equation}
 for some $x^*$ such that
\begin{equation}\label{supconvol.formula2.estimates}
 |x-x^*|^2\leq2\| \bar u\|_\infty\epsilon
\end{equation}
  (here we are  slightly abusing notation for the sake of brevity since, as $\bar u\in USC$,  we  should write $\sup\bar u$ instead of $\| \bar u\|_\infty$).
 To see this, first notice that
 since $\bar u^\epsilon$ is bounded above, for any given $\delta>0$ there exists $x_\delta$ such that,
\[
\bar u^\epsilon(x)= \sup_{y}
\left\{ \bar u(y)- \frac{|x-y|^2}{\epsilon}\right\} \leq  \bar u(x_\delta)- \frac{|x-x_\delta|^2}{\epsilon}+\delta.
\]
 Since $\bar u(x)\leq \bar u^\epsilon(x)$ (pick $y=x$ in the definition of $\bar u^\epsilon(x)$), we conclude
 that $|x - x_\delta|^2\leq(2\|\bar u\|_\infty+1)\epsilon$, assuming $\delta<1$. Therefore, 
   \[
 \bar u^\epsilon(x)\leq\sup_{|y-x|^2\leq(2\| \bar u\|_\infty+1)\epsilon}
\left\{ \bar u(y)- \frac{|x-y|^2}{\epsilon}\right\}+\delta.
 \]
Since $\delta$ is arbitrary, we can let $\delta\to0$ and conclude that the supremum in the definition of \eqref{sup.convolution} is achieved,
  \[
 \bar u^\epsilon(x)=\sup_{|y-x|^2\leq(2\| \bar u\|_\infty+1)\epsilon}
\left\{ \bar u(y)- \frac{|x-y|^2}{\epsilon}\right\}.
 \]
 At this point, we can repeat the 
previous argument with $\delta=0$ and get formula \eqref{supconvol.formula2}.

\end{enumerate}

\medskip

Now, again for the sake of contradiction,  assume  $\sup_{\mathbb{R}^n}(u-v)>0$. Notice that $\bar u^\epsilon(x)- \bar v_\epsilon(x)\geq\bar u(x)-\bar v(x)$  (pick $y=x$ in the definitions of $\bar u^\epsilon(x),\bar v_\epsilon(x)$), and therefore,
\begin{equation}\label{suprem.1}
\sup_{\mathbb{R}^n}(\bar u^\epsilon- \bar v_\epsilon)\geq
\sup_{\mathbb{R}^n}(\bar u- \bar v)=\sup_{\mathbb{R}^n}(u-v)>0.
\end{equation}
Moreover, $(\bar u^\epsilon- \bar v_\epsilon)(x)\to 0$ as $|x|\to \infty$. To see this, notice that
\[
\begin{split}
 \bar u (x)- \bar v (x)&\leq  \bar u ^\epsilon(x)- \bar v _\epsilon(x)
 \\
 &=\sup_{|y|^2\leq2\| \bar u\|_\infty\epsilon}
\left\{ \bar u(x+ y)- \frac{|y|^2}{\epsilon}\right\} - \inf_{|y|^2\leq2\| \bar v\|_\infty\epsilon}
\left\{ \bar v(x+ y)+ \frac{|y|^2}{\epsilon}\right\} 
\\
& \leq \sup_{|y|^2\leq2\| \bar u\|_\infty\epsilon} \bar u(x+y)- \inf_{|y|^2\leq2\| \bar v\|_\infty\epsilon} \bar v(x+y),
\end{split}
\]
 and   $ \bar u (x)- \bar v (x)$, $\sup_{|y|^2\leq2\| \bar u\|_\infty\epsilon} \bar u(x+y)$, and $ \inf_{|y|^2\leq2\| \bar v\|_\infty\epsilon} \bar v(x+y)$ converge to 0 as $|x|\to\infty$.

Thus, there exists $x_\epsilon$ such that
\begin{equation}\label{suprem.2}
 (\bar u^\epsilon- \bar v_\epsilon)(x_\epsilon)=\sup_{\mathbb{R}^n}(\bar u^\epsilon- \bar v_\epsilon).
\end{equation}

An important point in the sequel is that both functions $\bar u^\epsilon$ and $ \bar v_\epsilon$ are $\mathcal{C}^{1,1}$ at $x_\epsilon$, so that  the integrals in the  operators appearing in the subsequent computations are well-defined. This follows from the following three facts:
\begin{itemize}
 \item 
 The paraboloid 
  \[
  P(x)=\bar u(x_\epsilon^*)- \frac{|x-x_\epsilon^*|^2}{\epsilon}
  \]
   touches $\bar u^\epsilon$ from below at $x_\epsilon$
  for $x_\epsilon^*$ such that
  $\bar u^\epsilon(x_\epsilon)= \bar u(x_\epsilon^*)- \frac{|x_\epsilon-x_\epsilon^*|^2}{\epsilon}.$
  \item The paraboloid 
  \[
  Q(x)=\bar v(x_{\epsilon,*})+ \frac{|x-x_{\epsilon,*}|^2}{\epsilon}
  \]
   touches $\bar v_\epsilon$ from above at $x_\epsilon$
  for $x_{\epsilon,*}$ such that
  $\bar v_\epsilon(x_\epsilon)= \bar v(x_{\epsilon,*})+ \frac{|x_\epsilon-x_{\epsilon,*}|^2}{\epsilon}.$
 
   \item Since $x_\epsilon$ is a maximum point of $\bar u^\epsilon- \bar v_\epsilon$, the function $\bar v_\epsilon(x)-\bar v_\epsilon(x_\epsilon)+\bar u^\epsilon(x_\epsilon)$ touches $\bar u^\epsilon$ from above at $x_\epsilon$.
\end{itemize}
We conclude from these three facts that the paraboloids
$Q(x)- \bar v_\epsilon(x_\epsilon)+\bar u^\epsilon(x_\epsilon)$ and  $P(x)+ \bar v_\epsilon(x_\epsilon)-\bar u^\epsilon(x_\epsilon)$
touch respectively $\bar u^\epsilon$  from above and $\bar v_\epsilon$ from below   at the point $x_\epsilon$. Therefore, both $\bar u^\epsilon$ and $\bar v_\epsilon$ can  be touched from above and below by a paraboloid at $x_\epsilon$ and they are $\mathcal{C}^{1,1}$ at $x_\epsilon$.

The fact that both  $\bar u^\epsilon$ and $\bar v_\epsilon$  are $\mathcal{C}^{1,1}$ at $x_\epsilon$ is crucial to make rigorous the formal argument described at the beginning of the proof.  Since $\bar u^\epsilon\in\mathcal{C}^{1,1}$ at $x_\epsilon$, there exists a paraboloid $P(x)$ that touches $\bar u^\epsilon$ from above at $x_\epsilon$. Then, the function
\[
\tilde P(x)=P(x+x_{\epsilon}-x_{\epsilon}^*)+\frac{|x_{\epsilon}-x_{\epsilon}^*|^2}{\epsilon}+\phi(x)
\]
touches $ u$ from above at $x_{\epsilon}^*$. On the other hand, there exists a paraboloid $Q(x)$ that touches $\bar v_\epsilon$ from below at $x_\epsilon$ and then, the function
\[
\tilde Q(x)=Q(x+x_{\epsilon}-x_{\epsilon,*})-\frac{|x_{\epsilon}-x_{\epsilon,*}|^2}{\epsilon}+\phi(x)
\]
touches $v$ from below at $x_{\epsilon,*}$.   By Lemma \ref{viscosidad.clasico} we have
\[
\mathcal{D}_s  u(x_\epsilon^*)\geq g\big(x_\epsilon^*,u(x_\epsilon^*)\big),
\qquad \text{and} \qquad
\mathcal{D}_s v(x_{\epsilon,*})\leq g\big(x_{\epsilon,*},v(x_{\epsilon,*})\big)
\]
in the classical sense.

Fix $\eta>0$, arbitrary, and let $A_\eta>0$ with $\det A_\eta=1$ such that 
\[
L_{A_\eta} v(x_{\epsilon,*})\leq  g\big(x_{\epsilon,*},v(x_{\epsilon,*})\big)+\eta,
\]
and
\[
L_{A_\eta} u(x_\epsilon^*)\geq \mathcal{D}_s  u(x_\epsilon^*)\geq g\big(x_\epsilon^*,u(x_\epsilon^*)\big),
\]
with $L_{A_\eta}$  defined as in \eqref{defn.L_A}.
Subtracting, we get
\begin{equation}\label{comparisonson.relations.0}
 L_{A_\eta} u(x_\epsilon^*)-L_{A_\eta} v(x_{\epsilon,*})\geq g\big(x_\epsilon^*,u(x_\epsilon^*)\big)-g\big(x_{\epsilon,*},v(x_{\epsilon,*})\big)- \eta.
\end{equation}
The rest of the proof is devoted to derive a contradiction from the previous inequality by showing that, for $\epsilon$ small enough, the left-hand side is strictly smaller than the right-hand side.

Let us prove first  that,
\begin{equation}\label{the.big.thing}
\lim_{\epsilon\to\infty}\big(L_{A_\eta} u(x_\epsilon^*)-L_{A_\eta} v(x_{\epsilon,*})\big)\leq0.
 \end{equation}
By definition of the operator $L_{A_\eta}$, we have
\begin{equation}\label{relation.comp.operators.left}
L_{A_\eta} u(x_\epsilon^*)-L_{A_\eta} v(x_{\epsilon,*})=\frac{1}{2}
\int_{\mathbb{R}^n}\frac{\delta(u,x_\epsilon^*,y)- \delta(v,x_{\epsilon,*},y)}{|A_\eta^{-1}y|^{n+2s}}\, dy.
\end{equation}
Notice that
\begin{equation}\label{relation.deltas.comp.1}
  \delta(\bar u^\epsilon,x_\epsilon,y)\geq   \delta(\bar u,x_\epsilon^*,y)\qquad\text{and}\qquad
 \delta(\bar v_\epsilon,x_\epsilon,y)\leq   \delta(\bar v,x_{\epsilon,*},y).
\end{equation}
Since the proof of both inequalities is analogous, let us show how to obtain the first one.
As we have seen, 
\[
 \bar u^\epsilon(x_\epsilon)= \bar u(x_\epsilon^*)- \frac{|x_\epsilon-x_\epsilon^*|^2}{\epsilon}.
  \]
On the other hand, picking $z=x_\epsilon^*-x_\epsilon$,
\[
\bar u ^\epsilon(x_\epsilon\pm y) =\sup_{z}
\left\{ \bar u(x_\epsilon\pm y+z)- \frac{|z|^2}{\epsilon}\right\}\geq  \bar u(x_\epsilon^*\pm y)- \frac{|x_\epsilon-x_\epsilon^*|^2}{\epsilon}.
\]
From these two expressions, we get \eqref{relation.deltas.comp.1}.

Now, using \eqref{relation.deltas.comp.1}, we have that
\[
\begin{split}
 \delta(u,x_\epsilon^*,y)- \delta(v,x_{\epsilon,*},y)&=
  \delta(\bar u,x_\epsilon^*,y)- \delta(\bar v,x_{\epsilon,*},y)
  + \delta(\phi,x_\epsilon^*,y)- \delta(\phi,x_{\epsilon,*},y)\\
&  \leq
  \delta(\bar u^\epsilon,x_\epsilon,y)- \delta(\bar v_\epsilon,x_\epsilon,y)
  + \delta(\phi,x_\epsilon^*,y)- \delta(\phi,x_{\epsilon,*},y)\\
 &=  \delta(\bar u^\epsilon-\bar v_\epsilon,x_\epsilon,y)
  + \delta(\phi,x_\epsilon^*,y)- \delta(\phi,x_{\epsilon,*},y).
\end{split}
 \]
 Observe that  $x_\epsilon$ is a maximum point of $\bar u^\epsilon-\bar v_\epsilon$ and therefore $\delta(\bar u^\epsilon-\bar v_\epsilon,x_\epsilon,y)\leq0$. We conclude
 \begin{equation}\label{relation.deltas.comp.2}
 \delta(u,x_\epsilon^*,y)- \delta(v,x_{\epsilon,*},y)\leq
  \delta(\phi,x_\epsilon^*,y)- \delta(\phi,x_{\epsilon,*},y).
 \end{equation}
 
 From \eqref{comparisonson.relations.0}, \eqref{relation.comp.operators.left} and \eqref{relation.deltas.comp.2}, we get
 \begin{equation}\label{relation.comp.operators.left2}
 \begin{split}
 L_{A_\eta} \phi(x_\epsilon^*)&-L_{A_\eta} \phi(x_{\epsilon,*})
=\frac{1}{2}
\int_{\mathbb{R}^n}\frac{\delta(\phi,x_\epsilon^*,y)- \delta(\phi,x_{\epsilon,*},y)}{|A_\eta^{-1}y|^{n+2s}}\, dy\\
&\geq\frac{1}{2}
\int_{\mathbb{R}^n}\frac{\delta(u,x_\epsilon^*,y)- \delta(v,x_{\epsilon,*},y)}{|A_\eta^{-1}y|^{n+2s}}\, dy
=L_{A_\eta} u(x_\epsilon^*)-L_{A_\eta} v(x_{\epsilon,*})\\
&\geq  g\big(x_\epsilon^*,u(x_\epsilon^*)\big)-g\big(x_\epsilon^*,v(x_{\epsilon,*})\big)+g\big(x_\epsilon^*,v(x_{\epsilon,*})\big)-g\big(x_{\epsilon,*},v(x_{\epsilon,*})\big)- \eta.
\end{split}
\end{equation}

Recall from  \eqref{suprem.1} and \eqref{suprem.2} that
\[
 (\bar u^\epsilon- \bar v_\epsilon)(x_\epsilon)=\sup_{\mathbb{R}^n}(\bar u^\epsilon- \bar v_\epsilon)\geq
\sup_{\mathbb{R}^n}(\bar u- \bar v)>0.
\]
Therefore,
\[
\bar u(x_\epsilon^*)- \bar v(x_{\epsilon,*}) \geq
\sup_{\mathbb{R}^n}(\bar u- \bar v)+\frac{|x_\epsilon-x_\epsilon^*|^2+|x_\epsilon-x_{\epsilon,*}|^2}{\epsilon},
\]
or equivalently,
\[
u(x_\epsilon^*)- v(x_{\epsilon,*}) \geq
\sup_{\mathbb{R}^n}(\bar u- \bar v)+\big(\phi(x_\epsilon^*)- \phi(x_{\epsilon,*}) \big)+\frac{|x_\epsilon-x_\epsilon^*|^2+|x_\epsilon-x_{\epsilon,*}|^2}{\epsilon},
\]
Notice that from estimate \eqref{supconvol.formula2.estimates} and its analogous for the inf-convolution, we have
\[
 |x_\epsilon-x_\epsilon^*|^2\leq2\| \bar u\|_\infty\epsilon\qquad\text{and}\qquad  |x_\epsilon-x_{\epsilon,*}|^2\leq2\| \bar v\|_\infty\epsilon.
\]
Thus, by the continuity of $\phi$, we have that for $\epsilon$ small enough, 
\[
u(x_\epsilon^*)- v(x_{\epsilon,*}) \geq
\frac{1}{2}\sup_{\mathbb{R}^n}(\bar u- \bar v)>0.
\]
Since $\phi\in\mathcal{C}^{2,\alpha},$ in particular $L_{A_\eta} \phi(x)$ is a continuous function and, for $\epsilon$ small enough
\[
\big( L_{A_\eta} \phi(x_\epsilon^*)-L_{A_\eta} \phi(x_{\epsilon,*})\big)\leq\eta.
\]
By the continuity of $g$, we can also assume that $g\big(x_\epsilon^*,v(x_{\epsilon,*})\big)-g\big(x_{\epsilon,*},v(x_{\epsilon,*})\big)\geq- \eta$ .
Then, we have from \eqref{relation.comp.operators.left2} and \eqref{monotonicity.condition.g} that
\begin{equation}\label{final.knot}
3\eta\geq  g\big(x_\epsilon^*,u(x_\epsilon^*)\big)-g\big(x_\epsilon^*,v(x_{\epsilon,*})\big)\geq\mu\big(u(x_\epsilon^*)- v(x_{\epsilon,*})\big) \geq
\frac{\mu}{2}\sup_{\mathbb{R}^n}(\bar u- \bar v)>0.
\end{equation}
Since $\eta$ is arbitrary, we can choose $\eta\leq \frac{\mu}{12}\sup_{\mathbb{R}^n}(\bar u- \bar v)$ and get a contradiction.
\end{proof}

\medskip

\section{Lipschitz continuity and semiconcavity of solutions}\label{sect3.exist.lip}

In this section, we prove Lipschitz continuity and semiconcavity of solutions to \eqref{main.problem.g}
with $\phi$ under the hypothesis of Section \ref{section.intro}. These results are needed to fulfill the hypotheses of Theorem \ref{main.result.nondegeneracy.monge.ampere}.

\begin{remark}
The regularity results below apply to the operator $(1-s)\mathcal{D}_s$. Notice that all constants involved in the estimates are independent of $s$ and allow passing to the limit as $s\to1$. 
\end{remark}

We start with the particular case when $g(x,v(x))=v(x)- \phi(x)$ to illustrate the key ideas.

\begin{proposition}
Assume $\phi$ is semiconcave and Lipschitz continuous and let $v$ be the solution of
\begin{equation}\label{main.problem.v-phi}
\left\{
\begin{split}
&\mathcal{D}_sv(x)=v(x)- \phi(x)  \qquad \text{in}
\
\mathbb{R}^n \\
&(v- \phi)(x)\to0\quad\!\text{as}\ |x|\to\infty.
\end{split}
\right.
\end{equation}
 Then, $v$  is Lipschitz continuous and  semiconcave with the same constants as $\phi$.
\end{proposition}

\begin{proof}
In the following proof, we assume for clarity of presentation that $v$ is a classical solution to \eqref{main.problem.v-phi} and all the equations hold pointwise. The argument can be made rigorous using a regularization argument (similar to the one in the proof of Theorem \ref{them.comparison}) that is explained in detail in the proofs of the more general results Propositions \ref{lemma2} and \ref{lemma3} below, so we shall skip it here.

\medskip

1. For the proof of Lipschitz continuity, fix $e\in\mathbb{R}^n$ and consider  the first-order incremental quotient $v(x+e)-v(x)$.
Observe that
\[
v(x+e)-v(x)= (v-\phi)(x+e)-(v-\phi)(x)+\phi(x+e)-\phi(x)\leq o(1)+\textnormal{Lip}(\phi)\,|e| 
\]
as $|x|\to\infty$, and therefore $v(x+e)-v(x)$ is bounded above. Furthermore, we can assume  that 
\[
\sup_{x\in\mathbb{R}^n}\big(v(x+e)-v(x)\big)>\textnormal{Lip}(\phi)\,|e|, 
\]
since we are done otherwise. Then, there exists some  $x_0$ such that
\[
v(x_0+e)-v(x_0)=\sup_{x\in\mathbb{R}^n}\big(v(x+e)-v(x)\big).
\]

Fix $\eta>0$, arbitrary, and let $A_\eta>0$ such that 
\[
L_{A_\eta} v(x_0)\leq v(x_0)-\phi(x_0)+\eta,
\]
and
\[
L_{A_\eta} v(x_0+e)\geq \mathcal{D}_s  v(x_0+e)\geq v(x_0+e)-\phi(x_0+e),
\]
with $L_{A_\eta}$  defined as in \eqref{defn.L_A}. We have from the above expressions that
\[
L_{A_\eta} \big(v(x_0+e)- v(x_0)\big)\geq \big(v(x_0+e)- v(x_0)\big)-\big(\phi(x_0+e)- \phi(x_0)\big)-\eta.
\]
Notice that $\delta\big(v(\cdot+e)-v,x_0,y\big)\leq0$, and therefore $L_{A_\eta}\big(v(x_0+e)-v(x_0)\big)\leq0$.
Consequently,
\[
\sup_{x\in\mathbb{R}^n}\big(v(x+e)-v(x)\big)= v(x_0+e)- v(x_0)\leq \textnormal{Lip}(\phi)|e|+\eta
\]
and we conclude letting $\eta\to0$.

A symmetric argument, where $x_0$ is a point such that
\[
v(x_0+e)-v(x_0)=\inf_{x\in\mathbb{R}^n}\big(v(x+e)-v(x)\big)<-\textnormal{Lip}(\phi)\,|e|, 
\]
and
the operator $L_{A_\eta}$ is such that,
\[
L_{A_\eta} v(x_0+e)\leq v(x_0+e)-\phi(x_0+e)+\eta,
\]
and
\[
L_{A_\eta} v(x_0)\geq v(x_0)-\phi(x_0)
\]
yields
\[
\inf_{x\in\mathbb{R}^n}\big(v(x+e)-v(x)\big)\geq-\textnormal{Lip}(\phi)\,|e|.
\]

\medskip

2. For the proof of semiconcavity, consider  the second-order incremental quotient $\delta(v,x,e)=v(x+e)+v(x-e)-2v(x)$.
Denote by $SC(\phi)$ the semiconcavity constant of $\phi$, and notice that
\[
\delta(v,x,e)= \delta(v-\phi,x,e)+\delta(\phi,x,e)\leq o(1)+SC(\phi)\,|e|^2 \qquad\text{as}\ |x|\to\infty
\]
so $\delta(v,x,e)$ is bounded above. Furthermore, we can assume  that 
\[
\sup_{x\in\mathbb{R}^n}\delta(v,x,e)>SC(\phi)\,|e|^2
\]
since we are done otherwise. Then, there exists some  $x_0$ such that
\[
\delta(v,x_0,e)=\sup_{x\in\mathbb{R}^n} \delta(v,x,e).
\]
As before, fix $\eta>0$ arbitrary, and let $A_\eta>0$  such that 
\[
L_{A_\eta} v(x_0)\leq v(x_0)- \phi(x_0)+\eta,
\]
and
\[
L_{A_\eta} v(x_0\pm e)\geq \mathcal{D}_s  v(x_0\pm e)\geq v(x_0\pm e)- \phi(x_0\pm e),
\]
with $L_{A_\eta}$  defined as in \eqref{defn.L_A}.
We have from the above expressions that
\[
L_{A_\eta}\delta(v,x_0,e)\geq\delta(v,x_0,e)- \delta(\phi,x_0,e)-2\eta.
\]
Notice that $\delta\big(\delta(v,\cdot \, ,e),x_0,z\big)\leq0$, and therefore $L_{A_\eta}\delta(v,x_0,e)\leq0$.
Consequently,
\[
\delta(v,x,e)\leq\delta(v,x_0,e)\leq\delta(\phi,x_0,e)+2\eta\leq SC|e|^2+2\eta.
\]
We conclude letting $\eta\to0$.
\end{proof}

\medskip

In the next result we prove that solutions to \eqref{main.problem.g}  are Lipschitz continuous whenever $g$ on the right-hand side satisfies \eqref{Lipschitz.condition.g} and \eqref{monotonicity.condition.g}. 

\begin{proposition}[Lipschitz continuity  of the solution]\label{lemma2}
Let $g:
\mathbb{R}^{n+1}
\to\mathbb{R}$ satisfy  \eqref{Lipschitz.condition.g} and \eqref{monotonicity.condition.g}.
Then, $v$, the solution to \eqref{main.problem.g},
is uniformly Lipschitz continuous, namely, for every $x,y\in\mathbb{R}^n$,\
\[
\frac{|v(x)-v(y) |}{|x-y|}\leq \max\left\{\frac{\textnormal{Lip}(g)}{\mu}, \textnormal{Lip}(\phi)\right\}.
\]
\end{proposition}

\begin{proof}
The following proof uses a regularization process similar to the proof of Theorem \ref{them.comparison}. For the sake of clarity, let us  present first the main ideas assuming that $v$ is a classical solution. 

Fix $e\in\mathbb{R}^n$ and consider  the first-order incremental quotient $v(x+e)-v(x)$.
Observe that
\[
v(x+e)-v(x)= (v-\phi)(x+e)-(v-\phi)(x)+\phi(x+e)-\phi(x)\leq o(1)+\textnormal{Lip}(\phi)\,|e| 
\]
as $|x|\to\infty$, and therefore $v(x+e)-v(x)$ is bounded above. Furthermore, we can assume  that 
\[
\sup_{x\in\mathbb{R}^n}\big(v(x+e)-v(x)\big)>\textnormal{Lip}(\phi)\,|e|, 
\]
since we are done otherwise. Then, there exists some  $x_0$ such that
\[
v(x_0+e)-v(x_0)=\sup_{x\in\mathbb{R}^n}\big(v(x+e)-v(x)\big).
\]

Fix $\eta>0$, arbitrary, and let $A_\eta>0$ with $\det A_\eta=1$ such that 
\[
L_{A_\eta} v(x_0)\leq g\left(x_0,v(x_0)\right)+\eta,
\]
and
\[
L_{A_\eta} v(x_0+e)\geq \mathcal{D}_s  v(x_0+e)\geq g\left(x_0+e,v(x_0+e)\right),
\]
with $L_{A_\eta}$  defined as in \eqref{defn.L_A}.

We have from the above expressions that
\[
L_{A_\eta} v(x_0+e)-L_{A_\eta} v(x_0)\geq g\left(x_0+e,v(x_0+e)\right)-g\left(x_0,v(x_0)\right)-\eta.
\]
Notice that $\delta\big(v(\cdot+e)-v,x_0,y\big)\leq0$, and therefore $L_{A_\eta}\big(v(x_0+e)-v(x_0)\big)\leq0$.
Consequently,
\[
g\left(x_0+ e,v(x_0+ e)\right)-g\left(x_0,v(x_0)\right)\pm g\left(x_0+e,v(x_0)\right)\leq \eta.
\]
At his point we can let $\eta\to0$ and, using \eqref{Lipschitz.condition.g} and \eqref{monotonicity.condition.g},  get 
\begin{equation}\label{first.ineq.lip.proof}
v(x+e)-v(x) \leq  v(x_0+e)-v(x_0) \leq \frac{\textrm{Lip}(g)}{\mu} |e|. 
\end{equation}

A symmetric argument, where $x_0$ is a point such that
\[
v(x_0+e)-v(x_0)=\inf_{x\in\mathbb{R}^n}\big(v(x+e)-v(x)\big)<-\textnormal{Lip}(\phi)\,|e|, 
\]
and
the operator $L_{A_\eta}$ is such that,
\[
L_{A_\eta} v(x_0+e)\leq g\left(x_0+e,v(x_0+e)\right)+\eta,
\]
and
\[
L_{A_\eta} v(x_0)\geq g\left(x_0,v(x_0)\right)
\]
yields
\[
g\left(x_0,v(x_0)\right)-g\left(x_0+ e,v(x_0+ e)\right)\pm g\left(x_0+e,v(x_0)\right)\leq 0. 
\]
and from there,
\[
 -\frac{\textrm{Lip}(g)}{\mu} |e|\leq  v(x_0+e)-v(x_0) \leq v(x+e)-v(x) .
\]

\medskip

In general, in the above argument we cannot guarantee that $v$ is regular enough so that  both 
$L_{A_\eta} v(x_0+e)$ and $L_{A_\eta} v(x_0)$ are well-defined and the corresponding equations hold in the classical sense.

To complete the argument, we are going to  use a regularization process similar to the one in the proof of Theorem \ref{them.comparison}. Let us show the details in the proof of \eqref{first.ineq.lip.proof}.

To simplify the notation in the sequel, let us denote $u(x)=v(x+e)$ and  consider the sup- and inf-convolution of $ u, $ and $ v$, respectively,
\[
 u^\epsilon(x)=\sup_{y}
\left\{  u(y)- \frac{|x-y|^2}{\epsilon}\right\} =\sup_{y}
\left\{  v(y+e)- \frac{|x-y|^2}{\epsilon}\right\} 
\]
and 
\[
 v_\epsilon(x)=\inf_{y}
\left\{  v(y)+ \frac{|x-y|^2}{\epsilon}\right\}.
\]

In the proof of Theorem \ref{them.comparison} we were dealing with the regularization of $v-\phi$, a bounded function. In our case, $v$ is not bounded but its growth at infinity is controlled by $\phi$, which allows to prove the following:
\begin{enumerate}
 \item $u^\epsilon(x)$  is bounded above. Specifically,  there exists a constant $C>0$ depending only on $\phi$ and $\|v-\phi\|_\infty$ such that   $u^\epsilon(x)\leq C(1+|x+e|)$.
 To see this, notice that by our hypotheses on $\phi$,
 \[
 \phi(x)\leq a|x|^{-\epsilon}+\Gamma(x)\leq a|x|^{-\epsilon}+b|x|\leq a+b|x|
 \]
 for $|x|$ large enough, where $b$ depends on the convexity of the sections of $\Gamma$. Since $\phi$ is bounded near 0, we conclude that $ \phi(x)\leq a+b|x|$ for all $x$, maybe for a different constant $a$. Since $v-\phi$ is bounded,
 \[
\begin{split}
u^\epsilon(x)&=\sup_{y}
\left\{  (v-\phi)(y+e)+\phi(y+e)- \frac{|x-y|^2}{\epsilon}\right\}\\
& \leq
\sup_{y}
\left\{  \|v-\phi\|_\infty+a+b|y+e|- \frac{|x-y|^2}{\epsilon}\right\} \leq\\
& \leq
\|v-\phi\|_\infty+a+b|x+e|+
\sup_{y}
\left\{  b|x-y|- \frac{|x-y|^2}{\epsilon}\right\} \leq\\
& \leq
\|v-\phi\|_\infty+a+b|x+e|+
b^2\epsilon\leq C(1+|x+e|).
\end{split}
 \]

 \item As a consequence, the supremum in the definition of $u^\epsilon(x)$ is finite, 
 and for any given $\delta>0$ there exists
 $x_\delta$ such that,
\[
 u^\epsilon(x)= \sup_{y}
\left\{ u(y)- \frac{|x-y|^2}{\epsilon}\right\} \leq  u(x_\delta)- \frac{|x-x_\delta|^2}{\epsilon}+\delta.
\]

 \item The supremum in the definition of $u^\epsilon$ is achieved. In fact,
\[
  u^\epsilon(x)\leq\sup_{|y-x|\leq\sqrt\epsilon R}
\left\{  u(y)- \frac{|x-y|^2}{\epsilon}\right\}
= u(x^*)- \frac{|x-x^*|^2}{\epsilon}
 \]
 for some $x^*$ such that
$|x-x^*|\leq\sqrt\epsilon R$, where $R$ depends on $\textnormal{Lip}(\phi)$ and $ \|v-\phi\|_\infty$ but can be chosen independent of $\epsilon$ and $x$.

 To see this, fix $\delta<1$ and notice that 
$
u(x)\leq u^\epsilon(x) \leq   u(x_\delta)- \frac{|x-x_\delta|^2}{\epsilon}+\delta.
$
We conclude 
\[
\begin{split}
\frac{|x - x_\delta|^2}{\epsilon}&\leq (v-\phi)(x_\delta+e)- (v-\phi)(x+e)+\textnormal{Lip}(\phi)\,|x_\delta-x|+\delta\\
&\leq 2\|v-\phi\|_\infty+\textnormal{Lip}(\phi)\,|x_\delta-x|+1.
\end{split}
\]
From this expression, it follows that $|x - x_\delta|<\sqrt\epsilon R$ for some $R$ as before ($\sqrt\epsilon R$ is basically the larger root of the quadratic polynomial in $|x - x_\delta|$).
Therefore, 
   \[
  u^\epsilon(x)\leq\sup_{|y-x|\leq\sqrt\epsilon R}
\left\{  u(y)- \frac{|x-y|^2}{\epsilon}\right\}+\delta.
 \]
Since $\delta$ is arbitrary, we can let $\delta\to0$ and conclude that the supremum in the definition of $  u^\epsilon$ is achieved.

\item Analogous properties hold for  $v_\epsilon$. Notice that property (1) is simpler,
\[
 v_\epsilon(x)=\inf_{y}
\left\{  v(y)+ \frac{|x-y|^2}{\epsilon}\right\}=\inf_{y}
\left\{  (v-\phi)(y)+\phi(y)+ \frac{|x-y|^2}{\epsilon}\right\}\geq\inf(v-\phi)>-\infty.
\]
\end{enumerate}

\medskip

We are ready now to complete the proof. Following the formal argument above, we can assume that there exists $x_0$ such that
\[
v(x_0+e)-v(x_0)=\sup_{x\in\mathbb{R}^n}\big(v(x+e)-v(x)\big)>\textnormal{Lip}(\phi)\,|e|, 
\]

First, we need to prove that there exists $x_\epsilon$ such that
\[
(u^\epsilon-v_\epsilon)(x_\epsilon)=\sup_{x\in\mathbb{R}^n}(u^\epsilon-v_\epsilon).
\]
To see this, observe that
\[
\sup_{x\in\mathbb{R}^n}(u^\epsilon-v_\epsilon)\geq(u^\epsilon-v_\epsilon)(x_0)\geq(u-v)(x_0)=\sup_{x\in\mathbb{R}^n}\big(v(x+e)-v(x)\big)>\textnormal{Lip}(\phi)\,|e|.
\]
On the other hand,
\[
\begin{split}
 (u^\epsilon-v_\epsilon)(x)&=u(x^*)- \frac{|x-x^*|^2}{\epsilon}-v(x_*)- \frac{|x-x_*|^2}{\epsilon}\\
 &\leq (v-\phi)(x^*+e)-(v-\phi)(x_*)+\textnormal{Lip}(\phi)\,(2\sqrt\epsilon R+|e|).
\end{split}
\]
Therefore, for $\epsilon$ small enough,
$\textnormal{Lip}(\phi)\,(2\sqrt\epsilon R+|e|)< \sup_{x\in\mathbb{R}^n}(u^\epsilon-v_\epsilon)$
and
\[
 (u^\epsilon-v_\epsilon)(x)<o(1)+\sup_{x\in\mathbb{R}^n}(u^\epsilon-v_\epsilon)\qquad\textnormal{as}\  |x|\to\infty.
\]

Following the proof of Theorem \ref{them.comparison}, we can prove that  both $ u^\epsilon$ and $  v_\epsilon$ are $\mathcal{C}^{1,1}$ at $x_\epsilon$, so that  the integrals in the subsequent computations are well defined. The idea is that the paraboloids
\[
\frac{|x-x_{\epsilon,*}|^2}{\epsilon}+v(x_{\epsilon,*})-  v_\epsilon(x_\epsilon)+ u^\epsilon(x_\epsilon)
\]
and
\[
- \frac{|x-x_\epsilon^*|^2}{\epsilon}+  v_\epsilon(x_\epsilon)+u(x_\epsilon^*)- u^\epsilon(x_\epsilon)
\]
touch respectively $ u^\epsilon$  from above and $ v_\epsilon$ from below   at the point $x_\epsilon$.  Therefore, $ u^\epsilon$ and $ v_\epsilon$ can both be touched from above and below by a paraboloid at $x_\epsilon$ and they are $\mathcal{C}^{1,1}$ at $x_\epsilon$.

Since $ u^\epsilon\in\mathcal{C}^{1,1}$ at $x_\epsilon$, there exists a paraboloid $P(x)$ that touches $ u^\epsilon$ from above at $x_\epsilon$. Then
\[
P(x+x_{\epsilon}-x_{\epsilon}^*)+\frac{|x_{\epsilon}-x_{\epsilon}^*|^2}{\epsilon}
\]
touches $ u$ from above at $x_{\epsilon}^*$. Equivalently,
\[
P(x-e+x_{\epsilon}-x_{\epsilon}^*)+\frac{|x_{\epsilon}-x_{\epsilon}^*|^2}{\epsilon}
\]
touches $ v$ from above at $x_{\epsilon}^*+e$.
 On the other hand, there exists a paraboloid $Q(x)$ that touches $ v_\epsilon$ from below at $x_\epsilon$ and then
\[
Q(x+x_{\epsilon}-x_{\epsilon,*})-\frac{|x_{\epsilon}-x_{\epsilon,*}|^2}{\epsilon}
\]
touches $v$ from below at $x_{\epsilon,*}$.   By Lemma \ref{viscosidad.clasico} we have
\[
\mathcal{D}_s  v(x_\epsilon^*+e)\geq g\big(x_\epsilon^*+e,v(x_\epsilon^*+e)\big),
\qquad \text{and} \qquad
\mathcal{D}_s v(x_{\epsilon,*})\leq g\big(x_{\epsilon,*},v(x_{\epsilon,*})\big)
\]
in the classical sense.

Fix $\eta>0$, arbitrary, and let $A_\eta>0$ such that 
\[
L_{A_\eta} v(x_{\epsilon,*})\leq  g\big(x_{\epsilon,*},v(x_{\epsilon,*})\big)+\eta,
\]
and
\[
L_{A_\eta} v(x_\epsilon^*+e)\geq \mathcal{D}_s  v(x_\epsilon^*+e)\geq g\big(x_\epsilon^*+e,u(x_\epsilon^*+e)\big),
\]
with $L_{A_\eta}$  defined as in \eqref{defn.L_A}.
Subtracting, we get
\[
 L_{A_\eta} v(x_\epsilon^*+e)-L_{A_\eta} v(x_{\epsilon,*})\geq g\big(x_\epsilon^*+e,v(x_\epsilon^*+e)\big)-g\big(x_{\epsilon,*},v(x_{\epsilon,*})\big)- \eta.
\]
By definition of the operator $L_{A_\eta}$, we have
\[
L_{A_\eta} v(x_\epsilon^*+e)-L_{A_\eta} v(x_{\epsilon,*})=\frac{1}{2}
\int_{\mathbb{R}^n}\frac{\delta(u,x_\epsilon^*,y)- \delta(v,x_{\epsilon,*},y)}{|A_\eta^{-1}y|^{n+2s}}\, dy.
\]
Notice that, as in the proof of Theorem \ref{them.comparison},
\[
  \delta( u^\epsilon,x_\epsilon,y)\geq   \delta( u,x_\epsilon^*,y)\qquad\text{and}\qquad
 \delta( v_\epsilon,x_\epsilon,y)\leq   \delta( v,x_{\epsilon,*},y).
\]
 Observe that  $x_\epsilon$ is a maximum point of $ u^\epsilon- v_\epsilon$ and therefore $\delta( u^\epsilon-v_\epsilon,x_\epsilon,y)\leq0$. We conclude
 \[
 \begin{split}
\eta&\geq g\big(x_\epsilon^*+e,v(x_\epsilon^*+e)\big)-g\big(x_\epsilon^*+e,v(x_{\epsilon,*})\big)+g\big(x_\epsilon^*+e,v(x_{\epsilon,*})\big)-g\big(x_{\epsilon,*},v(x_{\epsilon,*})\big)\\
&\geq\mu\,\big(v(x_\epsilon^*+e)-v(x_{\epsilon,*})\big)-\textnormal{Lip}(g)\,|x_\epsilon^*+e-x_{\epsilon,*}|.\end{split}
 \]
 Notice that
 \[
v(x_\epsilon^*+e)-v(x_{\epsilon,*})\geq(u^\epsilon-v_\epsilon)(x_{\epsilon})=\sup_{\mathbb{R}^n}(u^\epsilon-v_\epsilon)\geq(u^\epsilon-v_\epsilon)(x_0)\geq(u-v)(x_0)=\sup_{\mathbb{R}^n}(u-v)
 \]
 Since $\eta$ is arbitrary, we can let $\eta\to0$ and get
 \[
 \begin{split}
\mu\sup_{x\in\mathbb{R}^n}\big(v(x+e)-v(x)\big)\leq \textnormal{Lip}(g)\,\big(|e|+|x_\epsilon^*-x_{\epsilon}|+|x_\epsilon-x_{\epsilon,*}|\big)\leq \textnormal{Lip}(g)\,\big(|e|+2\sqrt\epsilon R\big).
\end{split}
 \]
The result follows letting $\epsilon\to0$.
 \end{proof}

\medskip

In the next result we show that solutions to \eqref{main.problem.g}  are semiconcave, informally, that second derivatives of solutions to \eqref{main.problem.g}  are bounded from above, under certain conditions on  the right-hand side $g$. Before stating the result, let us identify  heuristically  the natural hypotheses on $g$ in our context if semiconcavity is expected from the solutions.

To simplify, consider instead of $\mathcal{D}_s$ a linear operator $L_A$ (defined as in \eqref{defn.L_A}) such that
\[
L_{A} v(x)= g\left(x,v(x)\right).
\]

Formally, we  have that $D_{ee}^2v(x_0)$ satisfies
\[
L_{A_\eta} D_{ee}^2v(x_0)=
\sum_{1\leq i,j\leq n} 
\partial^2_{x_ix_j}g(x_0,v(x_0))e_ie_j.
\]
where $\sum_{1\leq i,j\leq n} 
\partial^2_{x_ix_j}g(x_0,v(x_0))e_ie_j$ is the second derivative in the direction $e$, at the point $x_0$, of the composite function $x\mapsto g(x,v(x))$. Now, if $x_0$ is a maximum point of $D_{ee}^2v$ we get
\[
\sum_{1\leq i,j\leq n} 
\partial^2_{x_ix_j}g(x_0,v(x_0))e_ie_j=L_{A_\eta} D_{ee}^2v(x_0)\leq0.
\]
It can be checked that
\[
\begin{split}
\big[\partial^2_{x_ix_j}g(x,v(x))\big]_{1\leq i,j\leq n}=& 
\left[
\begin{array}{cc}
 I_{n\times n} & \nabla v (x)^t
\end{array}
\right]_{n\times (n+1)}
\left[\partial^2_{i,j}g(x,v(x)) \right]_{1\leq i,j\leq n+1}
\left[
\begin{array}{c}
 I_{n\times n} \\ \nabla v (x)
\end{array}
\right]_{(n+1)\times n}\\
&+\partial_{n+1}g(x,v(x)) \, D^2v(x)
\end{split}
\]
where $\partial^2_{i,j}g(x,v(x))$ and $\partial_{n+1}g(x,v(x))$ denote derivatives of $g$ as a function of  $n+1$ variables  evaluated at the point $(x,v(x))$.
Writing $\xi=(
 e^t, \ \langle\nabla v (x),e\rangle
)^t$ for convenience, we have
\[
0 \geq \sum_{1\leq i,j\leq n} 
\partial^2_{x_ix_j}g(x_0,v(x_0))e_ie_j =
\sum_{1\leq i,j\leq n+1} \partial^2_{i,j}g(x_0,v(x_0))\xi_i\xi_j 
+\partial_{n+1}g(x_0,v(x_0)) \, D_{ee}^2v(x_0)
\]
or equivalently,
\[
\partial_{n+1}g(x_0,v(x_0)) \, D_{ee}^2v(x_0)\leq -\sum_{1\leq i,j\leq n+1} \partial^2_{i,j}g(x_0,v(x_0))\xi_i\xi_j .
\]
This inequality suggests that in order to get an upper bound on $D_{ee}^2v(x_0)$ it is natural to require  $D^2g\geq -C \, Id$ and $\partial_{n+1}g(x_0,v(x_0))>\mu>0$,  namely hypotheses \eqref{semiconvexity.condition.g} and \eqref{monotonicity.condition.g}, since then
\[
\mu\,D_{ee}^2v(x_0)\leq- \sum_{1\leq i,j\leq n+1} \partial^2_{i,j}g(x_0,v(x_0))\xi_i\xi_j \leq C\, |\xi|^2 
\leq C\, (1+|\nabla v(x_0) |^2).
\]
From here  we have the desired estimate as long as we can guarantee that $v$ is Lipschitz. In  Proposition \ref{lemma2} we proved that this is actually the case provided hypotheses \eqref{Lipschitz.condition.g}, and \eqref{monotonicity.condition.g} hold true.  

\medskip

In the following result we justify the heuristic argument above.

\begin{proposition}[Semiconcavity  of the solution]\label{lemma3}
Let $g:
\mathbb{R}^{n+1}
\to\mathbb{R}$ satisfy \eqref{semiconvexity.condition.g}, \eqref{Lipschitz.condition.g}, and \eqref{monotonicity.condition.g}.
Then, the solution to \eqref{main.problem.g}
is semiconcave, that is, for every $x\in\mathbb{R}^n$,\
\[
\delta(v,x,y) \leq \frac{C}{\mu} \left(1+\max\left\{\left(\frac{\textnormal{Lip}(g)}{\mu} \right)^2, \textnormal{Lip}(\phi)^2\right\}\right) |y|^2.
\]
\end{proposition}

\begin{proof}
Let $v$ be the solution to problem \eqref{main.problem.g}, $e\in\mathbb{R}^n$ fixed,
and assume that 
\[
\sup_{x\in\mathbb{R}^n} \delta(v,x,e)>0,
\]
as the result is trivial otherwise. We observe that $\delta(v,x,e)\to0$ as $|x| \to\infty$. To see this, notice first that $\delta(v,x,e)=\delta(v- \phi,x,e)+\delta(\phi,x,e)=o(1)+\delta(\phi,x,e)$  as $|x| \to\infty$. Also, by our hypotheses on $\phi$, we have that
\[
\frac{\delta(\phi,x,e)}{|e|^2}=O\left(\frac{1}{|x|}\right)\quad\text{as}\ |x| \to\infty.
\]
Therefore, there is some
$x_0$ such that
\begin{equation}\label{ineedthisonelater}
\delta(v,x_0,e)=\sup_{x\in\mathbb{R}^n} \delta(v,x,e)>0.
\end{equation}

To complete the proof we need a regularization process as in the proof of Proposition \ref{lemma2}. Again, let us  present the ideas first assuming that $v$ is a classical solution and all the equations hold pointwise.

Fix $\eta>0$ arbitrary, and let $A_\eta>0$  such that 
\[
L_{A_\eta} v(x_0)\leq g\left(x_0,v(x_0)\right)+\eta,
\]
and
\[
L_{A_\eta} v(x_0\pm e)\geq \mathcal{D}_s  v(x_0\pm e)\geq g\big(x_0\pm e, v(x_0\pm e)\big),
\]
with $L_{A_\eta}$  defined as in \eqref{defn.L_A}.
We have from the above expressions that
\[
L_{A_\eta}\delta(v,x_0,e)\geq g\left(x_0+ e,v(x_0+ e)\right)+g\left(x_0- e,v(x_0- e)\right)-2g\left(x_0,v(x_0)\right)-2\eta.
\]
Notice that $\delta\big(\delta(v,\cdot \, ,e),x_0,z\big)\leq0$, and therefore $L_{A_\eta}\delta(v,x_0,e)\leq0$.
Consequently,
\[
g\left(x_0+ e,v(x_0+ e)\right)+g\left(x_0- e,v(x_0- e)\right)-2g\left(x_0,v(x_0)\right)\leq2\eta.
\]
At this point we can let $\eta\to0$ and rewrite the resulting expression as
\[
g\big((x_0,v(x_0))+ \theta_2\big) - g\big((x_0,v(x_0))- \theta_1\big)
\leq
2g(x_0,v(x_0))
- g\big((x_0,v(x_0))+ \theta_1\big) - g\big((x_0,v(x_0))- \theta_1\big)
\]
for
$\theta_1=
\big(
e,v(x_0+e)-v(x_0)
\big)$
and
$\theta_2
=\big(
-e,v(x_0-e)-v(x_0)
\big)
$. Then, by \eqref{semiconvexity.condition.g} and \eqref{monotonicity.condition.g} we have
\[
\begin{split}
 \mu \,\delta(v,x_0,e)&\leq g\big(x_0-e,v(x_0-e)\big)-g\big(x_0-e,2v(x_0)-v(x_0+e)\big)
\\
&=
g\big((x_0,v(x_0))+ \theta_2\big)- g\big((x_0,v(x_0))- \theta_1\big)
\\
&\leq
2g(x_0,v(x_0))
- g\big((x_0,v(x_0))+ \theta_1\big) - g\big((x_0,v(x_0))- \theta_1\big)
\leq C| \theta_1|^2
\end{split}
\]
and therefore, for any $x\in\mathbb{R}^n$, 
\[
\delta(v,x,e) \leq\delta(v,x_0,e) \leq \frac{C}{\mu} \left(1+\left(\frac{v(x_0+e)-v(x_0)}{|e|}\right)^2\right) |e|^2.
\]
The result follows applying Proposition \ref{lemma2}. 

\medskip

To complete the proof in the general case, let us sketch the regularization procedure. The details follow the lines of the proof of Proposition \ref{lemma2}. To simplify the notation, let us denote $u(x)=v(x+e),$ $w(x)=v(x-e)$ and  consider the sup-convolution of $u,w$ and the inf-convolution of $ v$, namely,
\[
 u^\epsilon(x)=\sup_{y}
\left\{  u(y)- \frac{|x-y|^2}{\epsilon}\right\} =\sup_{y}
\left\{  v(y+e)- \frac{|x-y|^2}{\epsilon}\right\}=
  v(x^*+e)- \frac{|x-x^*|^2}{\epsilon},
\]
\[
 w^\epsilon(x)=\sup_{y}
\left\{  w(y)- \frac{|x-y|^2}{\epsilon}\right\} =\sup_{y}
\left\{  v(y-e)- \frac{|x-y|^2}{\epsilon}\right\}=
  v(x^{**}-e)- \frac{|x-x^{**}|^2}{\epsilon},
\]
and 
\[
 v_\epsilon(x)=\inf_{y}
\left\{  v(y)+ \frac{|x-y|^2}{\epsilon}\right\}=
  v(x_*)+ \frac{|x-x_*|^2}{\epsilon}.
\]
for some points $x^*,x^{**},$ and $x_*$ within a distance $\sqrt{\epsilon}R$ from $x$ (see property 3 in the proof of Proposition \ref{lemma2}).

Assume \eqref{ineedthisonelater}. Then, on the one hand, we have that
\[
\sup_{\mathbb{R}^n}(u^\epsilon+w^\epsilon-2v_\epsilon)\geq
u^\epsilon(x_0)+w^\epsilon(x_0)-2v_\epsilon(x_0)\geq
\delta(v,x_0,e)=\sup_{x\in\mathbb{R}^n} \delta(v,x,e)>0.
\]
On the other hand,
\[
\begin{split}
 u^\epsilon(x)+ w^\epsilon(x)-2 v_\epsilon(x)&\leq   (v-\phi)(x^*+e)+ (v-\phi)(x^{**}-e)- 2  (v-\phi)(x_*)\\
& +\big(\phi(x^*+e)-\phi(x+e)\big)+\big(\phi(x^{**}-e)-\phi(x-e)\big)\\
&- 2 \big(  \phi(x_*)-\phi(x)\big)+\delta(\phi,x,e)\\
&\leq o(1)+4\textrm{Lip}(\phi)\sqrt\epsilon R + O\left(\frac{1}{|x|}\right)
 \end{split}
\]
as $|x|\to\infty$.
Therefore, for $\epsilon$ small enough, there exists $x_\epsilon$ such that
\[
u^\epsilon(x_\epsilon)+w^\epsilon(x_\epsilon)-2v_\epsilon(x_\epsilon)=\sup_{\mathbb{R}^n}(u^\epsilon+w^\epsilon-2v_\epsilon).
\]

Now, consider the following three paraboloids:
 \[
  P(x)= u(x_\epsilon^*)- \frac{|x-x_\epsilon^*|^2}{\epsilon},
  \qquad
  Q(x)= v(x_{\epsilon,*})+ \frac{|x-x_{\epsilon,*}|^2}{\epsilon},
  \]
  and
\[
  R(x)= w(x_\epsilon^{**})- \frac{|x-x_\epsilon^{**}|^2}{\epsilon}
  \]
  Then, all three $u^\epsilon,w^\epsilon$, and $v_\epsilon$ are $\mathcal{C}^{1,1}$ at $x_\epsilon$. To see this, 
notice that
\begin{itemize}
\item    $P(x)$ touches $ u^\epsilon$ from below at $x_\epsilon$ and
\[
2Q(x)-R(x)+u^\epsilon(x_\epsilon)+w^\epsilon(x_\epsilon)-2v_\epsilon(x_\epsilon)
\]
touches from above.

\item   $Q(x)$   touches $ v_\epsilon$ from above at $x_\epsilon$ and
\[
\frac12P(x)+\frac12R(x)+ v_\epsilon(x_\epsilon)-\frac12u^\epsilon(x_\epsilon)-\frac12w^\epsilon(x_\epsilon)
\]
touches from below.

\item  $R(x)$ touches $ w^\epsilon$ from below at $x_\epsilon$  and
\[
2Q(x)-P(x)+u^\epsilon(x_\epsilon)+w^\epsilon(x_\epsilon)-2v_\epsilon(x_\epsilon)
\]
touches from above.

\end{itemize}
Then, there are three paraboloids that touch $v$ from above at $x_\epsilon^*+e$ and $x_\epsilon^{**}-e$, and from below at $x_{\epsilon*}$.    By Lemma \ref{viscosidad.clasico} we have
\[
\mathcal{D}_s  v(x_\epsilon^*+e)\geq g\big(x_\epsilon^*+e,v(x_\epsilon^*+e)\big),\qquad
\mathcal{D}_s  v(x_\epsilon^{**}-e)\geq g\big(x_\epsilon^{**}-e,v(x_\epsilon^{**}-e)\big),
\]
and
\[
\mathcal{D}_s v(x_{\epsilon,*})\leq g\big(x_{\epsilon,*},v(x_{\epsilon,*})\big)
\]
in the classical sense. Fix $\eta>0$, arbitrary, and let $A_\eta>0$  such that 
\[
L_{A_\eta} v(x_{\epsilon,*})\leq  g\big(x_{\epsilon,*},v(x_{\epsilon,*})\big)+\eta
\]
with $L_{A_\eta}$  defined as in \eqref{defn.L_A}. Then
\begin{multline*}
 L_{A_\eta} u(x_\epsilon^*)+ L_{A_\eta} w(x_\epsilon^{**})-2L_{A_\eta} v(x_{\epsilon,*})\\
 \geq g\big(x_\epsilon^*+e,v(x_\epsilon^*+e)\big)+g\big(x_\epsilon^{**}-e,v(x_\epsilon^{**}-e)\big)-2g\big(x_{\epsilon,*},v(x_{\epsilon,*})\big)-2 \eta.
\end{multline*}
As in the proof of Theorem \ref{them.comparison},
\[
  \delta( u,x_\epsilon^*,y)+\delta( w,x_\epsilon^{**},y)-2   \delta( v,x_{\epsilon,*},y)\leq  \delta( u^\epsilon+w^\epsilon-2v_\epsilon,x_\epsilon,y)\leq 0.
\]
Since $\eta$ is arbitrary, we conclude
 \[
 g\big(x_\epsilon^*+e,v(x_\epsilon^*+e)\big)+g\big(x_\epsilon^{**}-e,v(x_\epsilon^{**}-e)\big)-2g\big(x_{\epsilon,*},v(x_{\epsilon,*})\big)\leq0.
 \]
Rearranging terms, we get,
\begin{multline*}
 g\big(x_\epsilon^{**}-e,v(x_\epsilon^{**}-e)\big)
 \pm g\big(x_{\epsilon}^{**}-e,2v(x_{\epsilon,*})-v(x_\epsilon^*+e)\big)
-g\big(2x_{\epsilon,*}-x_\epsilon^*-e,2v(x_{\epsilon,*})-v(x_\epsilon^*+e)\big)
\\
\leq
2g\big(x_{\epsilon,*},v(x_{\epsilon,*})\big)-g\big(x_\epsilon^*+e,v(x_\epsilon^*+e)\big)-g\big(2x_{\epsilon,*}-x_\epsilon^*-e,2v(x_{\epsilon,*})-v(x_\epsilon^*+e)\big). 
\end{multline*}
Let us analyze the left-hand side of the inequality first. By \eqref{monotonicity.condition.g} and \eqref{Lipschitz.condition.g}, we have 
\begin{multline*}
 g\big(x_\epsilon^{**}-e,v(x_\epsilon^{**}-e)\big)
 \pm g\big(x_{\epsilon}^{**}-e,2v(x_{\epsilon,*})-v(x_\epsilon^*+e)\big)
-g\big(2x_{\epsilon,*}-x_\epsilon^*-e,2v(x_{\epsilon,*})-v(x_\epsilon^*+e)\big)
\\
\geq \mu(v(x_\epsilon^*+e)+v(x_\epsilon^{**}-e)-2v(x_{\epsilon,*}))-\textnormal{Lip}(g)|x_\epsilon^*+x_{\epsilon}^{**}-2x_{\epsilon,*}|
\end{multline*}
If we denote
$\theta=(x_\epsilon^*+e-x_{\epsilon,*},v(x_\epsilon^*+e)-v(x_{\epsilon,*}))$ the right-hand side becomes
\begin{multline*}
2g\big(x_{\epsilon,*},v(x_{\epsilon,*})\big)-g\big(x_\epsilon^*+e,v(x_\epsilon^*+e)\big)-g\big(2x_{\epsilon,*}-x_\epsilon^*-e,2v(x_{\epsilon,*})-v(x_\epsilon^*+e)\big)\\
=
2g\big(x_{\epsilon,*},v(x_{\epsilon,*})\big)-g\big((x_{\epsilon,*},v(x_{\epsilon,*}))+\theta)-g\big((x_{\epsilon,*},v(x_{\epsilon,*}))-\theta)\leq C|\theta|^2
\end{multline*}
where in the last step  we have used 
\eqref{semiconvexity.condition.g}.
Therefore,
\[
\mu(v(x_\epsilon^*+e)+v(x_\epsilon^{**}-e)-2v(x_{\epsilon,*}))\leq C|\theta|^2+\textnormal{Lip}(g)|x_\epsilon^*+x_{\epsilon}^{**}-2x_{\epsilon,*}|
\]
Observe that
\begin{multline*}
v(x_\epsilon^*+e)+v(x_\epsilon^{**}-e)-2v(x_{\epsilon,*})\geq(u^\epsilon+w^\epsilon-2v_\epsilon)(x_\epsilon)\\
\geq(u^\epsilon+w^\epsilon-2v_\epsilon)(x_0)\geq(u+w-2v)(x_0)=\delta(v,x_0,e)=\sup_{x\in\mathbb{R}^n}\delta(v,x,e).
\end{multline*}
On the other hand,
\[
|\theta|^2=|x_\epsilon^*+e-x_{\epsilon,*}|^2+|v(x_\epsilon^*+e)-v(x_{\epsilon,*})|^2\leq
(1+\textnormal{Lip}(v)^2)|x_\epsilon^*+e-x_{\epsilon,*}|^2\leq
(1+\textnormal{Lip}(v)^2)(|x_\epsilon^*-x_{\epsilon,*}|+|e|)^2
\]
Finally, recall that $|x_\epsilon^*-x_{\epsilon,*}|\leq2\sqrt{\epsilon}R$ and $|x_\epsilon^*+x_{\epsilon}^{**}-2x_{\epsilon,*}|\leq4\sqrt{\epsilon}R$. All the above together yields,
\[
\mu \sup_{x\in\mathbb{R}^n}\delta(v,x,e)
\leq C(1+\textnormal{Lip}(v)^2)(2\sqrt{\epsilon}R+|e|)^2+4\textnormal{Lip}(g)\sqrt{\epsilon}R.
\]
Now we can let $\epsilon\to0$ and apply Proposition \ref{lemma2} to conclude.
\end{proof}

\medskip

\section{Existence of solutions}\label{sect.existence.final.final}

In this section we prove existence of solutions to the problem
\begin{equation}\label{problem.comparison.intro.sect.exist}
 \left\{
\begin{split}
&\mathcal{D}_{s} u=u-\phi \qquad \text{in}
\quad
\mathbb{R}^n \\
&(u- \phi)(x)\to0\quad\!\text{as}\ |x|\to\infty.
\end{split}
\right.
\end{equation}
One could consider existence for more general problems with a right-hand side $g(x,u)$ as in \eqref{main.problem.g}. The arguments below would work under assumptions on $g$ that guarantee the existence of appropriate sub- and supersolutions as well as comparison (see Section \ref{Sect.comparison}).

\medskip

The idea to find a supersolution is that, by the definition of $\mathcal{D}_{s}$ as an infimum of linear operators, it is enough to have the appropriate inequality for just one of them.
\begin{lemma}\label{lemma.supersol.basic}
Denote  $g(x)=\min\{1,|x|^{-(2s+\tau)}\}$ and $u_F(x)=C_F\cdot|x|^{2s-n}$, the fundamental solution of $(-\Delta)^s$ for an appropriate constant $C_F$.  Then, there exist constants  $0<\tau<\min\{2s-1,n-2s\}$ and $M>0$ such that
$u=\phi+M\cdot \big(u_F*g\big)$
 satisfies
\begin{equation}\label{existence.supersolution}
 \left\{
\begin{split}
-&(- \Delta )^{s}u\leq c_{n,s}(u- \phi)\qquad \text{in}\ 
\mathbb{R}^n \\
&(u- \phi)(x)\to0\quad\!\text{as}\ |x|\to\infty. 
\end{split}
\right.
\end{equation}
\end{lemma}

\begin{proof}
We construct an upper barrier of the form $u=\phi+w$ as a potential. We start the construction of $w$ with $w_0=u_F* g_0$ where $g_0(x)=|x|^{-(2s+\tau)}$ for some small $0<\tau<n-2s$.
Since both $n-2s<n$ and $2s+\tau<n$ while $(n-2s)+(2s+\tau)>n$,
 there are constants $a_0,a_1>0$ such that
\begin{equation}\label{decay.w0}
a_0|x|^{-\tau}\leq w_0(x)\leq a_1|x|^{-\tau}\quad\text{as}\quad |x|\to\infty.
\end{equation}
Also, by construction, 
\[
(-\Delta)^sw_0(x)=g_0(x)=|x|^{-(2s+\tau)}.
\]

Notice that $w_0$ decays at infinity (and therefore   $(u-\phi)=w_0\to0$ as $|x|\to\infty$) but it is not bounded at $0$. Consequently, we truncate $g_0$ and define $g_1=\min\{1,g_0\}$ and $w_1=u_F*g_1$.

The function  $w_1$ is  bounded, still radially decreasing, and has the same decay as $w_0$. To prove the last assertion, first notice that
\begin{equation}\label{77.553.1}
w_1=u_F*g_1=w_0-u_F*(g_0-g_1). 
\end{equation}
The function $g_0-g_1$ is supported in the ball of radius one,  therefore
\[
\big(u_F*(g_0-g_1)\big)(x)=C_F\int_{B_1(0)}\left(|y|^{-(2s+\tau)}-1\right)|x-y|^{2s-n}\,dy.
\]
For $|y|\leq1$ and  $|x|>2$ we have $|y|<|x|/2$ and from there we deduce
\[
\left(\frac{2}{3}\right)^{n-2s}|x|^{2s-n}\leq|x-y|^{2s-n}\leq2^{n-2s}|x|^{2s-n}.
\]
On the other hand, $\tau<n-2s$ implies that
\[
\int_{B_1(0)}\left(|y|^{-(2s+\tau)}-1\right)\,dy 
\]
is constant and therefore
\begin{equation}\label{decay.g0-g1}
b_0|x|^{2s-n}\leq \big(u_F*(g_1-g_0)\big)(x)\leq b_1|x|^{2s-n}\quad\text{as}\quad |x|\to\infty
\end{equation}
form some constants $b_0,b_1>0$.
Again, since $\tau<n-2s$ \eqref{decay.w0}, \eqref{77.553.1}, and \eqref{decay.g0-g1} prove that $w_1$ has the same decay as $w_0$ as claimed.

 Moreover, there exist constants $A_0,A_1>0$ such that
\[
A_0\min\{1,|x|^{-\tau}\}\leq w_1(x)\leq A_1\min\{1,|x|^{-\tau}\}.
\]
%
%
Also, by construction, we have
\[
(-\Delta)^sw_1(x)=g_1(x)=\min\{1,|x|^{-(2s+\tau)}\}.
\]

As a third  and final step, we dilate $w_1$ to our final $w$. To this aim, define $w(x)=M\cdot w_1(x)$ for  some large constant $M$ to be chosen. Then, 
\[
(-\Delta)^sw=M\cdot g_1(x)=M\, \min\{1,|x|^{-(2s+\tau)}\}
\] 
and 
\[
w(x)\geq M A_0\min\{1,|x|^{-\tau}\}.
\]

We are ready to check that for an appropriate $M$ the function $u=\phi+w$ satisfies
\[
-(-\Delta)^su\leq c_{n,s} (u-\phi).
\]
Indeed,
\begin{equation}\label{final.loop0}
c_{n,s}(u-\phi)=c_{n,s}w\geq c_{n,s}M A_0\min\{1,|x|^{-\tau}\}
\end{equation}
where $c_{n,s},A_0$ are given and M is to be chosen. On the other hand,
\begin{equation}\label{final.loop1}
\begin{split}
-(-\Delta)^su(x)&=-(-\Delta)^sw(x)-(-\Delta)^s\phi(x)\\
&=-M\min\{1,|x|^{-(2s+\tau)}\}-(-\Delta)^s\phi(x)\leq -(-\Delta)^s\phi(x).
\end{split}
\end{equation}
From our hypotheses on $\phi,$ 
\[
-(-\Delta)^s\phi(x)=-(-\Delta)^s\Gamma(x)-(-\Delta)^s\eta(x)\leq    C\big(|x|^{1-2s}+|x|^{-(2s+\epsilon)}\big)\leq   C\,|x|^{1-2s}
\]
for $|x|$ large, while
$-(-\Delta)^s\phi(x)$ is bounded in every neighborhood of the origin. Therefore
\begin{equation}\label{final.loop2}
-(-\Delta)^s\phi(x)\leq  C\cdot\min\{1,|x|^{1-2s}\}
\end{equation}
for some constant $C>0$.

In view of \eqref{final.loop0}, \eqref{final.loop1}, and \eqref{final.loop2}, we only have to control $C\cdot\min\{1,|x|^{1-2s}\}$ by  $c_{n,s}M A_0\min\{1,|x|^{-\tau}\}$ to conclude. If $\tau<2s-1$, a  large value of $M$ does it.
\end{proof}

\begin{proposition}(Existence of solutions)\label{the.existt.ttheorem}
There exists a unique solution of 
\begin{equation}\label{approximate.problem.lemma.existence}
 \left\{
\begin{split}
&\mathcal{D}_{s} u=u- \phi \qquad \text{in}
\quad
\mathbb{R}^n \\
&(u- \phi)(x)\to0\quad\!\text{as}\ |x|\to\infty. 
\end{split}
\right.
\end{equation}
\end{proposition}

\begin{proof}
First, observe that $\phi$ is a subsolution to the problem, since by convexity $\delta(
\phi,x,y)\geq0$, and therefore $\mathcal{D}_{s} \phi(x)\geq0$.

To find a supersolution  notice that $-c_{n,s}^{-1}\,(- \Delta )^{s}$ is one of the operators that compete for the infimum in the definition of $\mathcal{D}_{s}$, and we know 
from Lemma \ref{lemma.supersol.basic} that there is a function $\bar u$ such that $-c_{n,s}^{-1}\,(- \Delta )^{s}\bar u\leq \bar u- \phi$ with the right ``boundary data at infinity", that is, $(\bar u-\phi)\to0$ as 
$|x|\to\infty$. We have,
\[
\mathcal{D}_{s} \bar u(x)\leq
-c_{n,s}^{-1}\,(- \Delta )^{s}\bar u(x)\leq \bar u- \phi.
\]
By comparison, see Theorem \ref{them.comparison}, $\phi\leq \bar u$.


Consider the following approximating problem, 
\[
 \left\{
\begin{split}
&\mathcal{D}_{s}^\epsilon u=u- \phi \qquad \text{in}
\quad
\mathbb{R}^n \\
&(u- \phi)(x)\to0\quad\!\text{as}\ |x|\to\infty,
\end{split}
\right.
\]
where $\mathcal{D}_{s}^\epsilon$ is  the following non-degenerate operator, 
\[
\begin{split}
\mathcal{D}_s^\epsilon u(x)&=\inf\bigg\{
\text{P.V.}\int_{\mathbb{R}^n}\frac{u(y)-u(x)}{|A^{-1}(y-x)|^{n+2s}}\,
\, dy\ \bigg|\ A>0,\ \det A=1,\  \epsilon\,Id<A<\epsilon^{-1}\,Id\bigg\}\\
&=\inf\bigg\{\frac{1}{2}
\int_{\mathbb{R}^n}\frac{u(x+y)+u(x-y)-2u(x)}{|A^{-1}y|^{n+2s}}\, dy\ \bigg|\ A>0,\ \det A=1,\  \  \epsilon\,Id<A<\epsilon^{-1}\,Id\bigg\}.
\end{split}
\]
Notice that  $\phi$ and $\bar u$ as above are respectively a sub- and supersolution for these approximating problems for every $\epsilon$.

Being uniformly elliptic, the approximating problems have a solution for every $\epsilon$. To show this, consider for every $k=1,2,\ldots$ the following uniformly elliptic Dirichlet problem
\[
 \left\{
\begin{split}
&\mathcal{D}_{s}^\epsilon u=u- \phi \qquad \text{in}
\quad B_k(0) \\
&u= \phi\hspace{55pt} \text{in}\ \mathbb{R}^n\setminus{}B_k(0).
\end{split}
\right.
\]
Then, for every $k$ there exists a unique solution $u_k$, which is regular (depending on $\epsilon$ but not in $k$), see  \cite{Caffarelli.Silvestre, Caffarelli.Silvestre2} and the references therein. By comparison (see \cite{Caffarelli.Silvestre}) we have
\[
\phi\leq u_{k_1}\leq u_{k_2}\leq\bar u\qquad\textrm{in}\ B_{k_1}(0)
\]
for every $k_1\leq k_2$ (notice that  $\phi$ and $\bar u$ are always a sub- and supersolution), that is, the sequence $u_k$ is monotone increasing. We conclude that the sequence converges locally uniformly to the solution of the approximating problem.

Now, observe that $\mathcal{D}_s u(x)\leq \mathcal{D}_s^{\epsilon_1} u(x)\leq \mathcal{D}_s^{\epsilon_2} u(x)$ for any $\epsilon_1\leq\epsilon_2$ since the infimum is smaller the larger the class of matrices. In particular,  we have that every $u_\epsilon$ is a supersolution of \eqref{approximate.problem.lemma.existence} and a supersolution of the approximating problem with every smaller $\epsilon$. Picking $\epsilon_k=1/k$ for $k=1,2,\ldots$ we have by comparison (Theorem \ref{them.comparison}) that
\[
\phi\leq \ldots \leq u_{\epsilon_k}\leq\ldots\leq u_{\epsilon_2}\leq u_{\epsilon_1} \leq\bar{u}.
\]

The arguments in Section \ref{sect3.exist.lip} imply that every $u_\epsilon$ is Lipschitz continuous with the same constant as $\phi$, hence uniformly in $\epsilon$.
Therefore, when $\epsilon$ goes to 0, the sequence $u_\epsilon$ converges monotonically and uniformly to the solution of problem \eqref{approximate.problem.lemma.existence}, which is unique by comparison.
\end{proof}

To conclude, we show that the right-hand side of problem \eqref{approximate.problem.lemma.existence} is positive and  Theorem \ref{main.result.nondegeneracy.monge.ampere} applies. Therefore, the operator remains uniformly elliptic and standard regularity results for uniformly elliptic equations, see \cite{Caffarelli.Silvestre, Caffarelli.Silvestre2} are available.

\begin{proposition}
 Let $u$ be the solution to problem \eqref{approximate.problem.lemma.existence}. Then, $u>\phi$ in $\mathbb{R}^n$.
\end{proposition}

\begin{proof}
As pointed out in the proof of Proposition \ref{the.existt.ttheorem} we have  $u\geq\phi$ in $\mathbb{R}^n$ by comparison, therefore we need to prove that the inequality is strict. 
We argue by contradiction and assume to the contrary that there is $x_0\in\mathbb{R}^n$ such that $u(x_0)=\phi(x_0)$.

Observe that then $\phi$ touches $u$ from below at $x_0$. We can replace $\phi$ by $\tilde\phi\in\mathcal{C}^{2,\alpha}(\mathbb{R}^n)$, also strictly convex in compact sets and asymptotic to some cone at infinity in such a way that $\tilde\phi$ touches $\phi$ (and also $u$) from below at $x_0$, and this is the only contact point.

Then we can use $\tilde\phi$ as a test in the definition of viscosity solution of problem \eqref{approximate.problem.lemma.existence} to get $\mathcal{D}_{s} \tilde\phi(x_0)\leq0$. On the other hand, $\mathcal{D}_{s} \tilde\phi(x_0)\geq0$  by convexity, and we conclude that $\mathcal{D}_{s} \tilde\phi(x_0)=0$.

Then we claim that there exist a direction $e\in\partial B_1(0)$ such that the one-dimensional fractional laplacian of the restriction of $\tilde\phi$ to the direction $e$ is non-positive, namely
$-\big(- \Delta \big)_{e}^s\tilde\phi(x_0)\leq0$.
This is a contradiction with the fact that $-\big(- \Delta \big)_{e}^s\tilde\phi(x_0)>0$ since $\tilde\phi$ is convex and non-constant.

To prove the claim, notice that for all $k=1,2,\ldots$ there exists $e_k\in\partial B_1(0)$ such that 
\begin{equation}\label{one.eq.i.need.inremark.positiv}
 -\big(- \Delta \big)_{e_k}^s\tilde\phi(x_0)=\frac12
\int_{\mathbb{R}} \frac{\delta(\tilde\phi,x_0,te_k)}{|t|^{1+2s}}\,dt\leq\frac{1}{k}
\end{equation}
(otherwise, there exists $\mu>0$ such that $-\big(- \Delta \big)_{e}^s\tilde\phi(x_0)>\mu$ uniformly in $e$ and we can argue as in Proposition \ref{second.part.main} to show that $\mathcal{D}_{s} \tilde\phi(x_0)>0$).

By the convexity and linear growth at infinity of $\tilde\phi$ we have that 
\[
0\leq \frac{\delta(\tilde\phi,x_0,te_k)}{|t|^{1+2s}}\leq \frac{\min\{2\,\textrm{Lip}(\tilde\phi)|t|,C|t|^2\}}{|t|^{1+2s}}\in L^1(\mathbb{R})
\]
independently of $k$. Passing to a subsequence if necessary, we can assume that $e_k\to e$ as $k\to\infty$ and then we can pass to the limit in \eqref{one.eq.i.need.inremark.positiv}
by the dominated convergence theorem. This concludes the proof of the claim.
\end{proof}

\begin{remark}
Another approach to show existence would be to solve a ``truncated problem"  with the
restriction $\epsilon Id < A <\epsilon^{-1} Id$ besides the condition $\det(A)=1$. For this problem solutions exist and are smooth (depending on $\epsilon$) from existing
theory (see \cite{Caffarelli.Silvestre,Caffarelli.Silvestre2}), but also semiconvex independently of $\epsilon$ (Section \ref{sect3.exist.lip})  and the proof that
the operator remains strictly non-degenerate (Section \ref{sect4.unif.elipticity}) applies directly to these
solutions for $\epsilon$ small enough. 
\end{remark}

\medskip

\appendix

\section{}

In this appendix we include for the reader's convenience the proof of the following fact, mentioned in the introduction.
\begin{lemma}\label{caract.determ}
 If $u$ is convex,  asymptotically linear and $1/2<s<1$, then
\[
\lim_{s\to1}\big((1-s)\,\mathcal{D}_s u(x)\big)=\frac{ \omega_{n}}{4}\cdot\det(D^2u(x))^{1/n}
\]
in the viscosity sense.
\end{lemma}
The proof of Lemma \ref{caract.determ} is a direct consequence of the following two results.
\begin{lemma}\label{caract.determ.2}
 If $u$ is asymptotically linear and $1/2<s<1$, then
\[
\lim_{s\to1}\big((1-s)\,\mathcal{D}_s u(x)\big)=
\frac{ \omega_{n}}{4n}\cdot
\inf\Big\{
 {\rm trace}\big(AA^tD^2u(x)\big)
 \big|\ A>0,\ \det A=1\Big\}
\]
in the viscosity sense.
\end{lemma}

\begin{lemma}\label{caract.determ.3}
Let $B$ be a symmetric and positive semidefinite matrix. Then,
\[
n\det(B)^{1/n}=\inf\{{\rm trace}(AA^tB)\, |\ \det A=1\}.
\]
\end{lemma}

We devote the rest of this appendix to the proof of Lemmas \eqref{caract.determ.2} and \eqref{caract.determ.3} (Lemma \eqref{caract.determ.3} is well-known, but we include a  proof for completeness).

\begin{proof}[Proof of Lemma \ref{caract.determ.2}]
We shall first consider the case when $u\in\mathcal{C}^2(\mathbb{R}^n)$ and then show how to adapt the arguments to the viscosity setting. To this aim, let $A>0$ such that $\det A=1$ and $0<\rho<R$ to be chosen later on. Then,
\begin{equation}\label{bunch.of.integrals}
\begin{split}
\int_{\mathbb{R}^n}\frac{\delta(u,x,z)}{|A^{-1}z|^{n+2s}}\, dz&=\int_{B_{\rho}(0)}\frac{\langle D^2u(x)Ay,Ay\rangle}{|y|^{n+2s}}\, dy+\int_{B_{\rho}(0)}\frac{\delta(u,x,Ay)- \langle D^2u(x)Ay,Ay\rangle}{|y|^{n+2s}}\, dy\\
  &+\int_{B_R(0)\setminus B_{\rho}(0)}\frac{\delta(u,x,Ay)}{|y|^{n+2s}}\, dy+\int_{\mathbb{R}^n\setminus B_R(0)}\frac{\delta(u,x,Ay)}{|y|^{n+2s}}\, dy.\\
\end{split}
\end{equation}
Now, we are going to analyze each  term on the right-hand side of \eqref{bunch.of.integrals}. First, notice that
\begin{equation}\label{integral.first.term}
\begin{split}
\int_{B_\rho(0)}\frac{\langle D^2u(x)Ay,Ay\rangle}{|y|^{n+2s}}\,dy&=\frac{\rho^{2-2s}}{(2-2s)}\,
\int_{\partial B_1(0)}\langle D^2u(x)Ay,Ay\rangle\,d\mathcal{H}^{n-1}(y)\\
&=\frac{ \rho^{2-2s}}{(2-2s)}\,|B_1(0) | \, {\rm trace}\big(AA^tD^2u(x)\big). 
\end{split}
\end{equation}
The last equality follows integrating by parts (notice that $y$ is the unit normal to $\partial B_1(0)$).

Fix $\epsilon>0$, small. Since $\delta(u,x,Ay)= \langle D^2u(x)Ay,Ay\rangle+o(|y|^2)$ as  $|y|\to0$, we have that
\[
|\delta(u,x,Ay)- \langle D^2u(x)Ay,Ay\rangle| \leq \epsilon |y|^2
\]
if $|y|\leq\rho$ with $\rho$ sufficiently small.
Thus,
\begin{equation}\label{integral.second.term}
\left|\int_{B_\rho(0)}\frac{\delta(u,x,Ay)- \langle D^2u(x)Ay,Ay\rangle}{|y|^{n+2s}}\, dy\right|\leq\frac{\epsilon\,\omega_n}{2-2s}\rho^{2-2s}. 
\end{equation}

On the other hand, 
\begin{equation}\label{integral.third.term}
 \left|\int_{B_R(0)\setminus B_{\rho}(0)}\frac{\delta(u,x,Ay)}{|y|^{n+2s}}\, dy\right|\leq \frac{4\omega_n}{2s} \|u\|_{L^\infty\big(B_R(0)\setminus B_{\rho}(0)\big)}
 \big(\rho^{-2s} -R^{-2s}\big).
\end{equation}

As for the last term in \eqref{bunch.of.integrals}, since $u$ is asymptotically linear, for $R>0$ large enough, there exists some constant $L>0$ such that $|\delta(u,x,Ay)|\leq 2L |Ay| \leq 2L  \lambda_{\max}^{1/2}(AA^t) |y| $. Therefore, 
\begin{equation}\label{integral.fourth.term}
\left|
\int_{\mathbb{R}^n\setminus B_R(0)}\frac{\delta(u,x,Ay)}{|y|^{n+2s}}\, dy
\right|
\leq \frac{2L\,  \lambda_{\max}^{1/2}(AA^t)\omega_{n}}{(2s-1)}R^{1-2s}.
\end{equation}
Collecting \eqref{bunch.of.integrals}--%
\eqref{integral.fourth.term}, we get
\begin{multline*}
2\,(1-s)\,\mathcal{D}_s u(x)\leq\frac{ \omega_{n}}{2n}\cdot {\rm trace}\big(AA^tD^2u(x)\big)\,\rho^{2-2s}+\frac{\epsilon\,\omega_n}{2}\rho^{2-2s}\\
 +\left(\frac{1-s}{2s}\right)\,4\,\omega_n\, \|u\|_{L^\infty\big(B_R(0)\setminus B_{\rho}(0)\big)}
 \big(\rho^{-2s} -R^{-2s}\big)+\left(\frac{1-s}{2s-1}\right) 2L \, \lambda_{\max}^{1/2}(AA^t)\omega_{n}R^{1-2s},
\end{multline*}
and then,
\[
\lim_{s\to1}\big((1-s)\,\mathcal{D}_s u(x)\big)\leq\frac{ \omega_{n}}{4n}\cdot {\rm trace}\big(AA^tD^2u(x)\big)+\frac{\epsilon\,\omega_n}{4}.
\]
Since  both $A$ and $\epsilon$ are arbitrary, we have
\begin{equation}\label{inequality.1.visco.caract}
\lim_{s\to1}\big((1-s)\,\mathcal{D}_s u(x)\big)\leq
\frac{ \omega_{n}}{4n}\cdot
\inf\Big\{
 {\rm trace}\big(AA^tD^2u(x)\big)
 \big|\ A>0,\ \det A=1\Big\}.
\end{equation}

On the other hand, from \eqref{bunch.of.integrals}--%
\eqref{integral.fourth.term} we also get
\[
\begin{split}
\frac{ \omega_{n}}{2n}\,\cdot& \, \rho^{2-2s}\cdot 
\inf\Big\{
 {\rm trace}\big(AA^tD^2u(x)\big) \big|\ A>0,\ \det A=1\Big\} \\
 &\leq (1-s)\int_{\mathbb{R}^n}\frac{\delta(u,x,z)}{|A^{-1}z|^{n+2s}}\, dz
 +\frac{\epsilon\,\omega_n}{2}\rho^{2-2s}\\
 &+\left(\frac{1-s}{2s}\right)\,4\,\omega_n\, \|u\|_{L^\infty\big(B_R(0)\setminus B_{\rho}(0)\big)}\big(\rho^{-2s} -R^{-2s}\big)
 +\left(\frac{1-s}{2s-1}\right) 2L \, \lambda_{\max}^{1/2}(AA^t)\omega_{n}R^{1-2s},
\end{split}
\]
Let $A_\epsilon>0$ with $\det A_\epsilon=1$ such that 
\[
\frac{1}{2}\int_{\mathbb{R}^n}\frac{\delta(u,x,z)}{|A_\epsilon^{-1}z|^{n+2s}}\, dy\leq \mathcal{D}_s u(x)+\epsilon.
\]
Then,
\[
\begin{split}
\frac{ \omega_{n}}{4n}\,\cdot&\, \rho^{2-2s}\cdot
\inf\Big\{
 {\rm trace}\big(AA^tD^2u(x)\big)
 \big|\ A>0,\ \det A=1\Big\} \\
& \leq(1-s)\mathcal{D}_s u(x)+(1-s)\epsilon
+\frac{\epsilon\,\omega_n}{4}\rho^{2-2s}\\
 &+\left(\frac{1-s}{s}\right)\omega_n\, \|u\|_{L^\infty\big(B_R(0)\setminus B_{\rho}(0)\big)}\big(\rho^{-2s} -R^{-2s}\big)
  +\left(\frac{1-s}{2s-1}\right) L\,  \lambda_{\max}^{1/2}(A_\epsilon A_\epsilon^t)\,\omega_{n}R^{1-2s}.
\end{split}
\]
Finally, letting first $s\to1$ and then $\epsilon\to0$, we conclude
\begin{equation}\label{inequality.2.visco.caract}
 \lim_{s\to1}\big((1-s)\,\mathcal{D}_s u(x)\big)\geq
\frac{ \omega_{n}}{4n}\cdot
\inf\Big\{
 {\rm trace}\big(AA^tD^2u(x)\big)
 \big|\ A>0,\ \det A=1\Big\}
\end{equation}
and therefore, the equality.

\medskip

To conclude, let us show how to adapt this argument to the viscosity setting. According to Definition \ref{defn.viscosity.solution}, whenever a function $\psi$ touches $u$ (from above or from below) at a point $x$ in the sense that
\begin{itemize}
\item $\psi(x) = u(x)$,
\item $\psi(y) > u(y)$ (resp. $\psi(y) < u(y)$) for every $y \in N \setminus \{x\}$,
\end{itemize}
where $N$ is a neighborhood of $x$ and $\psi\in C^2(\overline N)$, then we have to evaluate $\mathcal{D}_s v(x)$
for 
\[
 v (y):= \begin{cases}
               \psi(y) &\text{in } N \\
	       u(y) &\text{in } \mathbb{R}^n \setminus N,
              \end{cases}
\]
and consider the appropriate inequality in the corresponding equation. Therefore, the main difference when carrying out the above argument in the viscosity sense is that \eqref{bunch.of.integrals} reads
\begin{equation}\label{bunch.of.integrals.viscosity}
 \begin{split}
\int_{\mathbb{R}^n}\frac{\delta(v,x,z)}{|A^{-1}z|^{n+2s}}\, dz&=\int_{B_{\rho}(0)}\frac{\langle D^2\psi(x)Ay,Ay\rangle}{|y|^{n+2s}}\, dy+\int_{B_{\rho}(0)}\frac{\delta(\psi,x,Ay)- \langle D^2\psi(x)Ay,Ay\rangle}{|y|^{n+2s}}\, dy\\
  &+\int_{B_R(0)\setminus B_{\rho}(0)}\frac{\delta(v,x,Ay)}{|y|^{n+2s}}\, dy+\int_{\mathbb{R}^n\setminus B_R(0)}\frac{\delta(u,x,Ay)}{|y|^{n+2s}}\, dy,
\end{split}
\end{equation}
provided $\rho$ and $R$ are respectively small and large enough. Then, we can apply estimates  \eqref{integral.first.term}--%
\eqref{integral.fourth.term} to each of the terms in \eqref{bunch.of.integrals.viscosity}
and get the corresponding one-sided inequality (analogous to \eqref{inequality.1.visco.caract}, \eqref{inequality.2.visco.caract}) needed to read the equation in the viscosity sense, namely,
\[
\lim_{s\to1}\big((1-s)\,\mathcal{D}_s v(x)\big)\geq
\frac{ \omega_{n}}{4n}\cdot
\inf\Big\{
 {\rm trace}\big(AA^tD^2\psi(x)\big)
 \big|\ A>0,\ \det A=1\Big\}\qquad\textrm{(resp. $\leq$)}.\qedhere
\]
\end{proof}

\begin{proof}[Proof of Lemma \ref{caract.determ.3}]
Let $A$ with $\det A=1$. Then, the matrix $AA^t$ is symmetric and positive definite. Hence, the  inequality between the arithmetic and geometric means  yields,
\[
n\det(B)^{1/n}\leq{\rm trace}(AA^tB).
\]
As this is true for any $A$ with $\det A=1$, we deduce,
\[
n\det(B)^{1/n}\leq\inf\{{\rm trace}(AA^tB)\ |\ \det A=1\}.
\]
To derive the converse inequality, assume first that $B>0$. If we choose
\[
A=P_B\,{\rm diag}\left[\frac{\det(B)^{1/(2n)}}{\lambda_i(B)^{1/2}}\right]P_B^t
\]
with $P_B$ such that $B=P_B\,{\rm diag}(\lambda_i(B))P_B^t$
and $P_BP_B^t=I$, we get 
\[
n\det(B)^{1/n}={\rm trace}(AA^tB)=\min\{{\rm trace}(AA^tB)\ |\ \det A=1\}.
\]

Let us now consider the case when $B\geq0$. Since the result is trivial when $B=0$, we can assume that 0 is an eigenvalue of the matrix  $B$ with multiplicity $n-k<n$, that is,
\[
B=P_B
\left(
\begin{array}{ccc|cccc}
 \lambda_1(B) & \hdots  &0   &0&\hdots&0\\
  \vdots&  \ddots & \vdots   &\vdots&&\vdots\\
  0&\hdots   & \lambda_k(B)    &0&\hdots&0\\\hline
  0&\hdots&0&0&\hdots&0\\
  \vdots&&\vdots&\vdots&\ddots&\vdots\\
  0&\hdots&0&0&\hdots&0
\end{array}
\right)
P_B^t
\]
with $P_BP_B^t=I$. Fix $\epsilon>0$ and  define 
\[
A_\epsilon=P_B\,{\rm diag}(\lambda_i(A_\epsilon))\,P_B^t
\]
with $P_B$ the same  as before and
\[
\lambda_i(A_\epsilon)=
\left\{\begin{split}
&\left(\frac{\epsilon}{{\rm trace}(B)}\right)^\frac12\hspace{38pt}{\rm for}\ i=1,\ldots ,k\\
&\left(\frac{\epsilon}{{\rm trace}(B)}\right)^{- \frac{k}{2(n-k)}}\quad{\rm for}\ i=k+1,\ldots ,n.
\end{split}
\right.
\]
In this way, $A_\epsilon$ is  positive definite with  ${\rm det} (A_\epsilon)=1$, and
\[
{\rm trace}(A_\epsilon A_\epsilon^tB)=\sum_{i=1}^k \lambda_i(A_\epsilon)^2\lambda_i(B)=\epsilon.
\]
Since $\epsilon>0$ is arbitrary, we conclude that
\[
\inf\{{\rm trace}(AA^tB)\ |\ \det A=1\}=0=n\det(B)^{1/n}.\qedhere
\]
\end{proof}


\section{Estimates of the integral of the kernel on the sphere}

\begin{proposition}\label{estimate.kernel.on.sphere}
Let $\epsilon_j>0$, for $j=1,\ldots,k$. Then,
 \[
\frac{\omega_k}{k}
\leq
\frac{\prod_{j=1}^{k}\epsilon_j}{\sum_{i=1}^k\epsilon_{i}^{-2s}}
 \int_{\partial B_1^k(0)}\frac{1}{\left(\sum_{j=1}^{k}\epsilon_{j}^2x_j^2\right)^\frac{k+2s}{2}}\,d\mathcal{H}^{k-1}(x)
 \leq
\frac{\pi^{k/2}}{s\,\Gamma(2-s)\,\Gamma\big(\frac{k}{2}+s\big)}.
 \]
\end{proposition}
\begin{remark}Notice that
\[
\lim_{s\to1} \frac{\pi^{k/2}}{s\,\Gamma(2-s)\,\Gamma\big(\frac{k}{2}+s\big)}=\frac{\omega_k}{k}.
\]
\end{remark}

\begin{proof}
\noindent{}1.\quad 
Start  denoting
\[
 \int_{\partial B_1^k(0)}\frac{1}{\left(\sum_{j=1}^{k}\epsilon_{j}^2x_j^2\right)^\frac{k+2s}{2}}\,d\mathcal{H}^{k-1}(x)=C<\infty.
\]
Multiply both sides by $r^{1-2s}e^{-r^2}$ and integrate from 0 to $\infty$, to get
\[
\int_0^\infty \int_{\partial B_1^k(0)}\frac{r^2e^{-r^2}}{\left(\sum_{j=1}^{k}\epsilon_{j}^2(rx_j)^2\right)^\frac{k+2s}{2}}\,d\mathcal{H}^{k-1}(x) \, r^{k-1} \,dr=C\int_0^\infty r^{1-2s}e^{-r^2}\,dr=\frac{\Gamma(1-s)}{2}C.
\]
Using polar coordinates, we get
\begin{equation}\label{aa.bb.1}
C=\frac{2}{\Gamma(1-s)} \int_{\mathbb{R}^k}\frac{|x|^2e^{-|x|^2}}{\left(\sum_{j=1}^{k}\epsilon_{j}^2x_j^2\right)^\frac{k+2s}{2}}\,dx.
\end{equation}
Notice that our choice of the radial function in the numerator, namely $|x|^2e^{-|x|^2}$, makes  the integral finite and well-defined.

\medskip

\noindent{}2.\quad We prove first the lower estimate. A change of variables $y_j=\epsilon_jx_j$ in \eqref{aa.bb.1} yields
\begin{equation}\label{1.ab.ba.1}
\begin{split}
&\int_{\mathbb{R}^k}\frac{|x|^2e^{-|x|^2}}{\left(\sum_{j=1}^{k}\epsilon_{j}^2x_j^2\right)^\frac{k+2s}{2}}\,dx=
\frac{1}{\prod_{j=1}^{k}\epsilon_j}
\int_{\mathbb{R}^k}\frac{\left(\sum_{j=1}^{k}\epsilon_{j}^{-2}y_j^2\right)e^{-\left(\sum_{j=1}^{k}\epsilon_{j}^{-2}y_j^2\right)}}{|y|^{k+2s}}\,dy\\
&=\frac{1}{\prod_{j=1}^{k}\epsilon_j}
 \int_{\partial B_1^k(0)}\Big(\sum_{j=1}^{k}\epsilon_{j}^{-2}y_j^2\Big)\int_0^{\infty} e^{-r^2\left(\sum_{j=1}^{k}\epsilon_{j}^{-2}y_j^2\right)}r^{1-2s}\,dr\,d\mathcal{H}^{k-1}(y).
\end{split}
\end{equation}
Performing a change of variables $t=\big(\sum_{j=1}^{k}\epsilon_{j}^{-2}y_j^2\big)^{1/2}r$ we get
\begin{equation}\label{2.ab.ba.2}
\begin{split}
\int_{\partial B_1^k(0)}\Big(\sum_{j=1}^{k}\epsilon_{j}^{-2}y_j^2\Big)&\int_0^{\infty} e^{-r^2\left(\sum_{j=1}^{k}\epsilon_{j}^{-2}y_j^2\right)}r^{1-2s}\,dr\,d\mathcal{H}^{k-1}(y)
\\
&=
\int_0^{\infty} t^{1-2s}e^{-t^2}dt\,\int_{\partial B_1^k(0)}\Big(\sum_{j=1}^{k}\epsilon_{j}^{-2}y_j^2\Big)^{s}\,d\mathcal{H}^{k-1}(y)
\\
&=
\frac{\Gamma(1-s)}{2}
\int_{\partial B_1^k(0)}\Big(\sum_{j=1}^{k}\epsilon_{j}^{-2}y_j^2\Big)^{s}\,d\mathcal{H}^{k-1}(y).
\end{split}
\end{equation}
Collecting \eqref{aa.bb.1}, \eqref{1.ab.ba.1}, and \eqref{2.ab.ba.2} yields
\[
C=\frac{1}{\prod_{j=1}^{k}\epsilon_j}\int_{\partial B_1^k(0)}\Big(\sum_{j=1}^{k}\epsilon_{j}^{-2}y_j^2\Big)^{s}\,d\mathcal{H}^{k-1}(y).
\]
Then, Jensen inequality with weights $y_j^2$ (notice that $\sum_{j=1}^ky_j^2=|y|^2=1$) yields
\[
\Big(\sum_{j=1}^{k}\epsilon_{j}^{-2}y_j^2\Big)^{s}
\geq
\sum_{j=1}^{k}\epsilon_{j}^{-2s}y_j^2.
\]
Integrating by parts (note that $y$ is the unit normal) we have
\[
\begin{split}
C
&\geq
\frac{1}{\prod_{j=1}^{k}\epsilon_j}\int_{\partial B_1^k(0)}\Big(\sum_{j=1}^{k}\epsilon_{j}^{-2s}y_j^2\Big)\,d\mathcal{H}^{k-1}(y)
\\
&=
\frac{1}{\prod_{j=1}^{k}\epsilon_j}\sum_{i=1}^k\int_{ B_1^k(0)}\partial_{y_i}\Big(\sum_{j=1}^{k}\epsilon_{j}^{-2s}\delta_{ij}y_j\Big)\,dy=
\frac{1}{\prod_{j=1}^{k}\epsilon_j}\sum_{i=1}^k\epsilon_{i}^{-2s}\frac{\omega_k}{k}
.
\end{split}
\]

\medskip

\noindent{}3.\quad The upper estimate follows from the identity
\begin{equation}\label{rewriting.spherical.kernel.2}
 \begin{split}
 \int_{\partial B_1^k(0)}\frac{1}{\left(\sum_{j=1}^{k}\epsilon_{j}^2x_j^2\right)^\frac{k+2s}{2}}\,d\mathcal{H}^{k-1}(x)&\\
 =\frac{\pi^{k/2}}{\Gamma(1-s)\Gamma\big(\frac{k+2s}{2}\big)}&
\sum_{i=1}^k\frac{1}{\epsilon_{i}^{2s}}\int_0^\infty\frac{t^{\frac{k+2s}{2}-1}}{(1+t)\prod_{j= 1}^k\left(\epsilon_i^2+\epsilon_j^2t\right)^{1/2}}\,dt
\end{split}
\end{equation}
just realizing that
\[
 \int_0^\infty\frac{t^{\frac{k+2s}{2}-1}}{(1+t)\prod_{j= 1}^k\left(\epsilon_i^2+\epsilon_j^2t\right)^{1/2}}\,dt
\leq
\frac{1}{\prod_{j=1}^{k}\epsilon_j}\left(
\int_0^1t^{s-1}\,dt +\int_1^\infty t^{s-2}\,dt
\right)
=\frac{1}{s(1-s)\prod_{j=1}^{k}\epsilon_j}.
\]

\medskip

\noindent{}4.\quad To prove \eqref{rewriting.spherical.kernel.2}, we shall use the following formula that follows from changing variables in the definition of the Gamma function 
\[
\lambda^{-h}=\frac{1}{\Gamma(h)}\int_{0}^{\infty} z^{h-1}e^{-\lambda z}\,dz,\qquad\textrm{ for all $h,\lambda>0$.}
\]
Applying this formula with $h=\frac{k+2s}{2}$ and $\lambda=\sum_{i=1}^{k}\epsilon_{i}^2x_i^2$ we can rewrite the kernel as
\[
\left(\sum_{i=1}^{k}\epsilon_{i}^2x_i^2\right)^{-\frac{k+2s}{2}}=\frac{1}{\Gamma\big(\frac{k+2s}{2}\big)}\int_{0}^{\infty} z^{\frac{k+2s}{2}-1}e^{-z\sum_{j=1}^{k}\epsilon_{j}^2x_j^2}\,dz.
\]
We get,
\begin{equation}\label{aa.bb.2}
\begin{split}
\Gamma\left(\frac{k+2s}{2}\right)\int_{\mathbb{R}^k}\frac{|x|^2e^{-|x|^2}}{\left(\sum_{j=1}^{k}\epsilon_{j}^2x_j^2\right)^\frac{k+2s}{2}}\,dx
&=\int_{\mathbb{R}^k} \int_{0}^{\infty} |x|^2 z^{\frac{k+2s}{2}-1}e^{-|x|^2-z\sum_{j=1}^{k}\epsilon_{j}^2x_j^2}\,dz\,dx\\
&=
\sum_{i=1}^k \int_{0}^{\infty} \int_{\mathbb{R}^k}  x_i^2 e^{-\sum_{j=1}^{k}(1+z\epsilon_{j}^2)x_{j}^2}\,dx\,z^{\frac{k+2s}{2}-1}\,dz.
\end{split}
\end{equation}
A change of variables $y_j=(1+z\epsilon_j^2)^{1/2}x_j$ yields
\begin{equation}\label{aa.bb.3}
\begin{split}
&\int_{0}^{\infty} \int_{\mathbb{R}^k}  x_i^2 e^{-\sum_{j=1}^{k}(1+z\epsilon_{j}^2)x_{j}^2}\,dx\,z^{\frac{k+2s}{2}-1}\,dz\\
&\qquad=
 \int_{\mathbb{R}}  y_i^2 e^{-y_{i}^2} dy_i
\left(\int_{\mathbb{R}}  e^{-t^2}\,dt\right)^{k-1}
 \int_{0}^{\infty}\frac{z^{\frac{k+2s}{2}-1}}{(1+z\epsilon_i^2)^{3/2}\prod_{j\neq i}^{k}(1+z\epsilon_j^2)^{1/2}}\,dz\\
 &\qquad=
\frac{\pi^\frac{k}{2}}{2}
 \int_{0}^{\infty}\frac{z^{\frac{k+2s}{2}-1}}{(1+z\epsilon_i^2)^{3/2}\prod_{j\neq i}^{k}(1+z\epsilon_j^2)^{1/2}}\,dz.
\end{split}
\end{equation}
Identity \eqref{rewriting.spherical.kernel.2} follows from \eqref{aa.bb.1}, \eqref{aa.bb.2}, and \eqref{aa.bb.3} with the change of variables $t=\epsilon_i^2 z$.
\end{proof}

\medskip


\bibliographystyle{plain}

\end{document}